\documentclass[11pt]{article}
\usepackage{amsfonts}
\usepackage{lipsum}
\usepackage{latexsym}
\usepackage{amsmath}
\usepackage{amssymb}
\usepackage{amsthm}
\usepackage{graphicx}
\usepackage{dsfont}
\usepackage{fullpage}
\usepackage{enumerate}
\usepackage{color}
\usepackage{bbm}
\newtheorem{theorem}{Theorem}[section]
\newtheorem{lemma}[theorem]{Lemma}

\newtheorem{proposition}[theorem]{Proposition}
\newtheorem{corollary}[theorem]{Corollary}

\newtheorem{conjecture}[theorem]{Conjecture}

\theoremstyle{definition}
\newtheorem{definition}[theorem]{Definition}
\newtheorem{example}[theorem]{Example}

\newtheorem{thmy}{Theorem}
\newenvironment{oldtheorem}{\stepcounter{thm}\begin{thmy}}{\end{thmy}}
\newtheorem*{note*}{Note}

\newcommand{\R}{\mathbb R}
\newcommand{\Rn}{{\mathbb R}^n}

\newcommand{\Ha}{{{\cal H}^{n-1}}}

\makeatother

\newcommand\blfootnote[1]{%
  \begingroup
  \renewcommand\thefootnote{}\footnote{#1}%
  \addtocounter{footnote}{-1}%
  \endgroup
}

\begin{document}


\title{\bf On Ball's conjectured Santal\'o-type inequality}
\date{\today}
\medskip
\author{K\'aroly J. B\"or\"oczky, Konstantinos Patsalos, Christos Saroglou}
\maketitle
\begin{abstract}
We prove that if $K$ is a symmetric and isotropic convex body in $\mathbb{R}^n$, then
$$\int_K\langle x,u\rangle^2\,dx\int_{K^\circ}\langle x,u\rangle^2\,dx\leq \left(\int_{B_2^n}\langle x,u\rangle^2\,dx\right)^2,\qquad\forall u\in\mathbb{R}^n,$$with equality for some $u\neq o$, if and only if $K$ is a Euclidean ball. This confirms a conjecture by Keith Ball (1986), stating that for any symmetric convex body $K$ in $\mathbb{R}^n$, it holds
$$\int_K\int_{K^\circ}\langle x,y\rangle^2\,dx\,dy\leq \int_{B_2^n}\int_{B_2^n}\langle x,y\rangle^2\,dx\,dy,$$with equality if and only if $K$ is an ellipsoid.

 Fortunately, our method for proving Ball's conjectured inequality admits a  quantitative stability refinement, which in turn yields an asymptotically optimal stability version of the Blaschke-Santal\'o inequality for origin symmetric convex bodies in terms of the symmetric difference metric. This resolves another well known open problem.

\end{abstract}

\section{Introduction}

For each $n\in \mathbb{N}$, we fix an inner product $\langle \cdot ,\cdot\rangle$ and an orthonormal basis $\{e_1,\dots,e_n\}$ in $\R^n$. The origin will be denoted by $o$. For $x\in \R^n$, set $x_j:=\langle x,e_j\rangle$, $j=1,\dots,n$, and $\|x\|:=\sqrt{\langle x,x\rangle}$.
Set also $B_2^n:=\{x\in \R^n:\|x\|\leq 1\}$ and $S^{n-1}:=\{x\in\R^n:\|x\|=1\}$ to be the Euclidean unit ball and the Euclidean unit sphere, respectively, with respect to the inner product $\langle \cdot ,\cdot\rangle$ and denote $\omega_n=|B^n_2|$. Set $\R^n_+$ to be the first orthant, that is,
$\R^n_+=\{x\in\R^n:x_i\geq 0,\ i=1,\dots,n\}$. 
The inradius $r(K)$ of a convex set $K$ will denote the maximal radius of any Euclidean balls contained in $K$. 
A convex body $K$ will be a compact convex set with nonempty interior. 
In addition, $K$ will be called symmetric, if $K=-K$. Finally, $K$ will be called unconditional if for any $x=(x_1,\dots,x_n)\in K$ and for any choice $\varepsilon_1,\dots,\varepsilon_n\in\{\pm 1\}$, the point $(\varepsilon_1 x_1,\dots,\varepsilon_n x_n)$ is also contained in $K$.

For a symmetric convex body $K$ in $\R^n$, its polar body $K^\circ$ is defined by
$$
K^\circ=\{x\in\R^n:\langle x,y\rangle\leq 1\;\forall y\in K\}
$$
and satisfies
\begin{equation}
\label{polar-linear-map}
(T K)^\circ=T^{-t}K^\circ,\qquad \forall T\in GL(n)
\end{equation}
One of the cornerstones of convex geometry is the classical Blaschke-Santal\'o inequality, which we  only state here in the symmetric case. 

\begin{oldtheorem}[Blaschke-Santal\'o inequality]
\label{thm-old-BS} 
If $K$ is a symmetric convex body in $\R^n$, then
\begin{equation}
\label{eq-san}
|K|\cdot |K^\circ|\leq |B_2^n|^2,
\end{equation}
with equality if and only if $K$ is an ellipsoid. 
\end{oldtheorem}

The inequality for all $n\geq 2$ is due to Santal\'o \cite{San}.
By now, there are several proofs of the Blaschke-Santal\'o inequality and some extensions of it, see e.g. \cite[Theorem 7.3]{And06}, \cite{AKM04}, \cite{Bal}, \cite{FrM}, \cite{Leh09b}, \cite{Leh09c}, \cite{Lut91}, \cite{LZ}, \cite{MeP} and others.

Keith Ball \cite{Bal}, \cite{bal2} has conjectured the following Santal\'o-type inequality.

\begin{conjecture}[Ball, 1986]
\label{conj-ball}
If $K$ is a symmetric convex body in $\R^n$, then
\begin{equation}
\label{eq-conj}
    \int_K\int_{K^\circ}\langle x,y\rangle^2\, dx\, dy\leq \int_{B_2^n}\int_{ B^n_2}\langle x,y\rangle^2\, dx\, dy.
\end{equation}
\end{conjecture}

He proved that \eqref{eq-conj} is stronger than \eqref{eq-san}, in the sense that if \eqref{eq-conj} was known to be true, then \eqref{eq-san} would follow by the latter in a few lines.  

Ball proved Conjecture \ref{conj-ball} in the unconditional case. The key to his proof is the following.

\begin{oldtheorem}[Ball]\label{prop-Ball} If $X,Y\subseteq \R^n_+$ are compact sets of positive volume such that $\langle x,y\rangle\leq 1$ holds for all $x\in X$ and $y\in Y$, then for any $j=1,\dots,n$, it holds
\begin{equation}\label{eq-unc-ball}
\int_{X}x_j^2\, dx\int_{Y}x_j^2\, dx\leq \left(\int_{\R^n_+\cap B_2^n}x_j^2\, dx\right)^2=\frac{1}{2^{2n}}\left(\int_{ B_2^n}x_j^2\, dx\right)^2,   
\end{equation}
with equality if and only if there exists an unconditional ellipsoid $E$ such that $X=\R^n_+\cap E$ and $Y=\R^n_+\cap E^\circ$.
\end{oldtheorem}

To be more precise, \eqref{eq-unc-ball} was originally established (as part of the proof of the unconditional case of Conjecture \ref{conj-ball}) in \cite{Bal} and \cite{bal2} for $X=C\cap \R^n_+$ and $Y=C^\circ\cap \R^n_+$, where $C$ is an unconditional convex body. However, the extension to the more general setting is straightforward
(see e.g. Kalantzopoulos, Saroglou \cite[Proposition 2.1]{KaS} and the comments after it). We recall the argument for Theorem~\ref{prop-Ball}, together with the characterization of the equality case in Section~\ref{secBallTheorem}. 

Our main goal is to confirm Ball's conjecture (Conjecture \ref{conj-ball}) in full generality.

\begin{theorem}
\label{thm-Ball-conj}
If $K$ is a symmetric convex body in $\R^n$, then
    $$\int_K\int_{K^\circ}\langle x,y\rangle^2\, dx\, dy\leq \int_{B_2^n}\int_{ B^n_2}\langle x,y\rangle^2\, dx\, dy,$$
 with equality if and only if $K$ is an origin symmetric ellipsoid.
\end{theorem}

In fact, we prove a slightly stronger result in Theorem~\ref{thm-main}. In this paper, we say that a symmetric convex body $L\subseteq\R^n$ is isotropic  (see  Milman, Pajor \cite{Mi-Pa} or Artstein-Avidan, Giannopoulos, Milman \cite{AGM15,AGM21})  if 
\begin{equation}
\label{quasi-iso-u}
\int_L \langle x,u\rangle^2\,dx=\frac{\|u\|^2}{n}\int_{L}\|x\|^2 \,dx\qquad\forall u\in\R^n.
\end{equation}
Note that unlike in the usual definition of isotropicity, we do not require any normalization on $L$. 

\begin{theorem}
\label{thm-main}
Let $K\subseteq \R^n$ be a symmetric convex body, and let $u\in \R^n\setminus\{o\}$. If $K^\circ$ (or equivalently, if $K$) is isotropic, then 
\begin{equation}\label{eq-thm-main}
  \int_{K}\langle x,u\rangle^2\, dx\int_{K^\circ}\langle x,u\rangle^2\, dx\leq \left(\int_{B_2^n}\langle x,u\rangle^2\, dx\right)^2,
\end{equation}
with equality if and only if $K$ is a Euclidean ball centered at the origin.
\end{theorem}

Our proof of Theorem \ref{thm-main} is inspired by Lehec's proof \cite{Leh09b} of a functional version of the classical Blaschke-Santal\'o inequality, originally established in Fradelizi-Meyer \cite{FrM}. As an immediate corollary of Theorem \ref{thm-main}, one has
\begin{equation}
\label{eq-moment-product}
\int_{K}\|x\|^2\,dx\cdot \int_{K^\circ}\|x\|^2\,dx\leq \left(\int_{B_2^n}\|x\|^2\,dx\right)^2,
\end{equation}
where $K$ (or $K^\circ$) is a symmetric isotropic convex body and equality holds  if and only if $K$ is a Euclidean ball centered at the origin.

We mention that, as observed in \cite[Proof of Proposition 3.1]{HL} (see also \cite{KaS}), one can deduce a functional version of Theorem \ref{thm-Ball-conj} by Theorem \ref{thm-Ball-conj} itself. This is done by adapting the proof of \cite[Proposition 3]{FrM} to the setting of Conjecture \ref{conj-ball}.
\begin{corollary}
Let $\rho:\R\to\R_+$ and $f,g:\R^n\to\R_+$ be even measurable functions satisfying $f(x)g(y)\leq \rho^2(\langle x,y\rangle)$, for all $x,y\in\R^n$ such that $\langle x,y\rangle>0$. Then,
\begin{equation*}
\int_{\R^n}\int_{\R^n}\langle x,y\rangle^2f(x)g(y)\, dx\, dy\leq   n^{-1}\left(\int_{\R^n}\|u\|^2\rho\left(\|u\|^2\right)\,du\right)^2.
\end{equation*}
Equality holds if and only if there exists a continuous function $\tilde{\rho}:\R_{+}\to\R_{+}$, satisfying the following. 
\vskip 1mm\noindent
{\bf a.} $\rho=\tilde{\rho}$ a.e., $\sqrt{\tilde\rho(s)\tilde\rho(t)}\le\tilde\rho(\sqrt{st})$ for every $s,t\geq 0$ 
and if $n\geq 2$, then $\tilde\rho(0)>0$ or $\tilde\rho|_{\R_+}\equiv 0$.
\vskip 1mm \noindent
{\bf b.}  For some positive definite $n\times n$ matrix $T$ and for some $d>0$, one has 
$$f(x) = d\tilde\rho( |Tx|^2)\hbox{ and } g(x) =\frac{1}{d}\tilde\rho(|T^{-1}x|^2)\quad a.e.$$
\end{corollary}

Fusco, Maggi, Pratelli \cite{FMP08} proved an optimal stability version of the Isoperimetric inequality in terms of the symmetric difference metric, which result was extended to the Anisotropic Isoperimetric inequality (or equivalently, to the Brunn-Minkowski inequality) by Figalli, Maggi, Pratelli \cite{FMP10}.
For symmetric convex bodies $K,C\subseteq\R^n$, their homothetic distance, considered, for example, by
Figalli, Maggi, Pratelli \cite{FMP10} and 
 Fusco, Maggi, Pratelli \cite{FMP08}, is
$$
A(K,C)=|\alpha K\Delta \beta C|\mbox{ \ where }\alpha=|K|^{-\frac1n}\mbox{ and }\beta=|C|^{-\frac1n}.
$$
Here $|\alpha K|=|\beta C|=1$. It is easy to see that the Banach-Mazur distance
$$
d(K,B^n_2)=\min\{\log \lambda:\,\lambda\geq 1\mbox{ and }\exists\Phi\in{\rm GL}(n),\;\Phi K\subseteq B^n_2\subseteq \lambda\Phi K\}
$$
of $K$ and $B^n_2$ satisfies
\begin{equation}
\label{BM-Fraenkel}
d(K,B^n_2)\leq C_nA(K,E)^{\frac2{n+1}}
\end{equation}
for any centered ellipsoid $E$ where $C_n>0$ depends on $n$. We are ready to state the stability version Theorem~\ref{Ball-ineq-stab0}  of Theorem~\ref{thm-Ball-conj}.

\begin{theorem}
\label{Ball-ineq-stab0}
If $K\subseteq \R^n$ is a symmetric convex body, then there exists an $o$-symmetric ellipsoid $E\subseteq\R^n$ such that
$$\int_K\int_{K^\circ}\langle x,y\rangle^2\,dydx\leq \left(1-\theta_n A(K,E)^2\right) \int_{B^n_2}\int_{B^n_2}\langle x,y\rangle^2\,dydx$$
where $\theta_n>0$ is an explicit constant depending only on $n$.
\end{theorem}
The proof of Theorem ~\ref{Ball-ineq-stab0} combines the method leading to Theorem ~\ref{thm-main} and the statement of the theorem itself with novel ideas. Theorem~\ref{Ball-ineq-stab0} implies the stability of various additional linear invariant fundamental inequalities in terms of the symmetric difference metrics.
Stability versions of the Blaschke-Santal\'o inequality and the Affine Isoperimetric inequality (cf. \eqref{affine-iso}) have been known in terms of the "easier to handle" Banach-Mazur-distance (see Ball, B\"or\"oczky \cite{BaB11}, B\"or\"oczky \cite{Bor10} and Ivaki \cite{Iva15}), but this is the first instance, following the footsteps of Fusco, Maggi, Pratelli \cite{FMP08} and Figalli, Maggi, Pratelli \cite{FMP10}, to provide a stability version symmetric difference metrics, which is the most natural distance in this respect. 

\begin{theorem}
\label{BS-ineq-stab}
If $K\subseteq \R^n$ is a symmetric convex body, then there exists an $o$-symmetric ellipsoid $E\subseteq\R^n$ such that
$$|K|\cdot |K^\circ|\leq \left(1-\theta_n A(K,E)^2\right) \omega_n^2,$$
where $\theta_n>0$ is an explicit constant depending only on $n$.
\end{theorem}

Here the stability estimate with respect to the Banach-Mazur distance after Theorem~\ref{BS-ineq-stab} is an improvement on the best one previously known provided by Ball, B\"or\"oczky \cite{BaB11}.

 It is well-known (cf. Lutwak \cite{Lut93b}) that the Blaschke-Santal\'o inequality \eqref{eq-san} is equivalent to the Affine Isoperimetric inequality \eqref{affine-iso} (see also \cite{DoH95}, \cite{Lud10}, \cite{Lut91}, \cite{ScW90}). 
 For a symmetric convex body $K\subseteq\R^n$, 
 the Gaussian curvature $\kappa_K(x)\geq 0$ exists at $\mathcal{H}^{n-1}$ a.e. $x\in\partial K$ by Alexandrov's theorem, and the affine surface area 
 is defined as
\begin{align*}
\Omega(K)=&\int_{\partial K}\kappa_K(x)^{\frac1{n+1}}\,d\mathcal{H}^{n-1}(x)
\end{align*}
The Affine Isoperimetric inequality 
says that
\begin{equation}
\label{affine-iso}
\Omega(K)\leq n\omega_n^{\frac2{n+1}}|K|^{\frac{n-1}{n+1}}
\end{equation}
with equality 
if and only if $K$ is an ellipsoid. Now, Theorem~\ref{BS-ineq-stab} implies the first close to be optimal stability version of the Affine Isoperimetric inequality 
in terms of the symmetric volume difference.

\begin{corollary}
\label{Affine-ineq-stab}
If $K\subseteq \R^n$ is a symmetric convex body, then there exists an $o$-symmetric ellipsoid $E\subseteq\R^n$ such that
$$\Omega(K)\leq \left(1-\theta_n A(K,E)^2\right)^{\frac{1}{n+1}} n\omega_n^{\frac2{n+1}}|K|^{\frac{n-1}{n+1}}
$$
where $\theta_n>0$ is an explicit constant depending only on $n$.
\end{corollary}
\noindent{\bf Remark.} Concerning the Banach-Mazur distance, it follows from \eqref{BM-Fraenkel} that in Theorems \ref{Ball-ineq-stab0} and \ref{BS-ineq-stab} and in Corollary \ref{Affine-ineq-stab}, the term $\theta_nA(K,E)^2$ may be replaced by $\tilde{\theta}_nd(K,B_2^n)^{n+1}$ where $\tilde{\theta}_n>0$ depends on $n$. The corresponding stability estimates with respect to the Banach-Mazur distance are improvements on the best ones previously known, proved by Ball, B\"or\"oczky \cite{BaB11}.\\

This paper can be naturally divided into two parts. The first part (Sections \ref{secBallTheorem}, \ref{secYaoYao}, \ref{secBall-conj}) is devoted to the proof of Ball's conjecture (Theorems \ref{thm-main} and \ref{thm-Ball-conj}). The second part (in which some of the arguments of the first part are repeated), consisting of Sections \ref{secWeak-stab}, \ref{secBall-ineq-stab}, \ref{secSantalo}, proves the stability results of the paper.

\section{Proof of Theorem \ref{prop-Ball}}
\label{secBallTheorem}

The proof of Theorem \ref{prop-Ball} will be a consequence of the classical Pr\'ekopa-Leindler inequality, where we only quote the equality case in the case we need. This argument will have important role also in proving the stability version Theorem~\ref{BS-ineq-stab} of  Keith Ball's conjectured inequality.

\begin{theorem}[Pr\'ekopa-Leindler-Dubuc \cite{Pr}, \cite{Le}, \cite{Du}]
\label{PLn}
If $f,g,h:\,\R^n\to [0,\infty)$ are integrable functions on $\R^n$ with positive integral
satisfying $h((x+y)/2))\geq \sqrt{f(x)g(y)}$ for
any $x,y\in\mathbb{R}^n$, then
\begin{equation}
\label{PLnineq}
\left(\int_{\R^n}  h\right)^2\geq \left(\int_{\R^n}f\right) \cdot \left(\int_{\R^n}g\right).
\end{equation}
If, in addition, $h$ is log-concave, and equality holds in \eqref{PLnineq}, then
there exists $w\in \R^n$ and $a>0$, such that $f(x)=ah(x-w)$ and $g(x)=a^{-1}h(x+w)$ for a.e. $x\in\R^n$. 
\end{theorem}


For $X,Y\subseteq \R^n_+$, we consider the ``coordinatewise geometric mean"
$$
X^{\frac12}\cdot Y^{\frac12}=\{(\sqrt{x_1y_1},\dots,\sqrt{x_ny_n})\in \R^n_+:(x_1,\dots,x_n)\in X,\ (y_1,\dots,y_n)\in Y\}.
$$
We observe that if $\langle x,y\rangle\leq 1$ for any $x\in X$ and $y\in Y$, then
\begin{equation}
\label{XYinB2n}
X^{\frac12}\cdot Y^{\frac12}\subseteq B^n_2.
\end{equation}

\begin{proof}[Proof of Theorem \ref{prop-Ball}.] We may assume that $i=1$. Since $X\cap{\rm int}\,\R^n_+\neq\emptyset$ and $Y\cap{\rm int}\,\R^n_+\neq\emptyset$, the condition $\langle x,y\rangle\leq 1$ for any $x\in X$ and $y\in Y$ yields that both $X$ and $Y$ are bounded.
We consider the integrable functions $h,f,g:\,\R^n\to [0,\infty)$ defined as
\begin{align}
f(t_1,\dots,t_n)=&e^{2t_1}\mathbf{1}_X(e^{t_1},\dots,e^{t_n})\cdot e^{t_1+\dots+t_n},\label{f-function}\\
g(t_1,\dots,t_n)=&e^{2t_1}\mathbf{1}_Y(e^{t_1},\dots,e^{t_n})\cdot e^{t_1+\dots+t_n},\label{g-function}\\
h(t_1,\dots,t_n)=&e^{2t_1}\mathbf{1}_{B^n_2}(e^{t_1},\dots,e^{t_n})\cdot e^{t_1+\dots+t_n},\label{h-function}
\end{align}
and hence the substitution $(t_1,\dots,t_n)\mapsto(e^{t_1},\dots,e^{t_n})$ shows that these functions satisfy
\begin{equation}\label{integrals-substitut}
 \int_{\R^n}f(x)\, dx=\int_Xx_1^2\, dx,\qquad
\int_{\R^n}g(x)\, dx=\int_Yx_1^2\, dx,\qquad
\int_{\R^n}h(x)\, dx=\int_{B^n_2}x_1^2\,dx.   
\end{equation}

Moreover, \eqref{XYinB2n} and the condition $\langle x,y\rangle\leq 1$ for any $x\in X$ and $y\in Y$ ensure that
$$h\left((x+y)/2)\right)\geq \sqrt{f(x)g(y)},$$for any
$x,y\in\mathbb{R}^n$. Therefore,  the Pr\'ekopa-Leindler inequality \eqref{PLnineq} yields \eqref{eq-unc-ball}.

Assume now that equality holds in \eqref{eq-unc-ball}. Thus, equality holds in the Pr\'ekopa-Leindler inequality \eqref{PLnineq} for $f,g,h$. We deduce from Theorem~\ref{PLn} that there exist $a>0$ and $w=(w_1,\dots,w_n)\in\R^n$ such that  for $b_j=e^{w_j}$, $j=1,\dots,n$, we have
\begin{align*}
e^{2t_1}\mathbf{1}_X(e^{t_1},\dots,e^{t_n})\cdot e^{t_1+\dots+t_n}=&
a\cdot \frac{e^{2t_1}}{b_1^2}\cdot \mathbf{1}_{B^n_2}\left(\frac{e^{t_1}}{b_1},\dots,\frac{e^{t_n}}{b_n}\right)\cdot \frac{e^{t_1+\dots+t_n}}{b_1\dots b_n},\\
e^{2t_1}\mathbf{1}_Y(e^{t_1},\dots,e^{t_n})\cdot e^{t_1+\dots+t_n}=&
b_1^2\cdot \frac{e^{2t_1}}{a}\cdot \mathbf{1}_{B^n_2}\left(b_1e^{t_1},\dots,b_ne^{t_n}\right)\cdot e^{t_1+\dots+t_n}\cdot b_1\dots b_n.
\end{align*}
As the value of an indicator function is $0$ or $1$, and $X$ and $Y$ are compact, it follows that 
$$
X=\R^n_+\cap \Phi B^n_2 \qquad\textnormal{and}\qquad Y=\R^n_+\cap \Phi^{-1} B^n_2,
$$
for the diagonal transformation $\Phi$ with eigenvalues $b_1,\dots,b_n$.
This proves Theorem \ref{prop-Ball}.
\end{proof}

\section{Yao-Yao partitions}
\label{secYaoYao}

Let $E$ be an affine space of finite dimension. For the purposes of this paper, a {\it partition} of $E$ will be a finite family of non overlapping convex cones whose union is the whole $E$. A special family of partitions, called Yao-Yao partitions, was invented by Yao and Yao \cite{YY}. A Yao-Yao partition ${\cal P}$ (together with its {\it center}) is defined inductively as follows.

\begin{definition}\label{def-Yao-Yao}
Let $E$ be an affine subspace of a finite dimensional vector space $V$. If $\dim E=0$, i.e. $E=\{x\}$ for some element $x\in V$, the family ${\cal P}=\{\{x\}\}$ is defined to be a Yao-Yao partition of $E$ and $x$ is called the {\it center} of ${\cal P}$. When $\dim E=n\geq 1$, we say that ${\cal P}$ is a Yao-Yao partition  of $E$ with center $x\in E$, if there exists a hyperplane $F$ of $E$, a vector $v\in V$ which is not parallel to $F$ and two Yao-Yao partitions ${\cal P}_+$ and ${\cal P}_-$ of $F$ with the same center $x\in F$, such that 
$$
{\cal P}=\{A+\R_+v:A\in {\cal P}_+\}\cup \{A+\R_-v:A\in {\cal P}_-\}.
$$
In this case, we say that the Yao-Yao partition ${\cal P}$ {\it is based on the hyperplane $F$ of $E$}, and {\it$v$ is an axis of }${\cal P}$.
\end{definition}

Let us collect the basic properties of a Yao-Yao partition ${\cal P}$ of $\R^n$, that will be needed later (see \cite{Leh09a}). 
\begin{itemize}
    \item The center of ${\cal P}$ is well defined.
    \item ${\cal P}$ consists of $2^n$ elements of the form $x+\textnormal{pos}\,\{v_1,\dots,v_n\}$, where $x$ is the center of ${\cal P}$ and $v_1,\dots,v_n\in\R^n$ are linearly independent vectors.
    \item If the center of ${\cal P}$ is the origin, then the family ${\cal P}^*:=\{A^*:A\in {\cal P}\}$ is also a partition of $\R^n$ (see \cite{Leh09b}). Here, 
$$
A^*:=\{y\in\R^n:\langle x,y\rangle\geq 0, \ \forall x\in A\}
$$
is the dual cone of $A$.
    \item If $T:\R^n\to\R^n$ is a non-singular affine map, then the family $\{TA:A\in {\cal P}\}$ is also a Yao-Yao partition of $\R^n$.
\item If $H^+\subseteq E$ is a half-space containing the center of ${\cal P}$, then there exists a $A\in {\cal P}$ such that
\begin{equation}
\label{Yao-Yao-cone-in-halfspace}
A\subseteq H^+.
\end{equation}
\end{itemize}
Let $\mu$ be a finite Borel measure on $\R^n$. We say that a Yao-Yao partition {\cal P} is a {\it Yao-Yao equipartition} of $\mu$, if for all $A\in {\cal P}$, $\mu(A)=2^{-n}\mu(\R^n)$. Under some mild assumptions, the existence of a Yao-Yao equipartition for $\mu$ was established by Yao and Yao \cite{YY} and later by Lehec \cite{Leh09a} (under weaker assumptions on $\mu$).
Lehec's construction appears to provide more information than the original construction by Yao and Yao. 
For example, let $\lambda_1,\dots,\lambda_n:\R^n\to\R$ be affine forms, such that the map 
$$
E\in z\mapsto (\lambda_1(z),\dots,\lambda_n(z))\in\R^n
$$ 
is one to one. Following \cite[Definition 6]{Leh09a}, we say (inductively) that a Yao-Yao partition ${\cal P}$ based on a hyperplane $F\subseteq\R^n$ is {\it adapted to the coordinate system} $(\lambda_1,\dots,\lambda_n)$ if $F=\{y\in E:\lambda_1(y)=\lambda_1(x)\}$ and ${\cal P}_+$, ${\cal P}_-$ are both adapted to $(\lambda_2|_F,\dots,\lambda_n|_F)$.
According to Lehec's \cite[Proposition 14]{Leh09a}, if $\mu$ is a finite Borel measure on $\R^n$ such that $\mu(H)=0$ for all hyperplanes $H\subseteq\R^n$ and $\textnormal{supp}\,\mu=\R^n$, then there exists a Yao-Yao equipartition of $\mu$ that is adapted to $(\lambda_1,\dots,\lambda_n)$, and two such Yao-Yao partitions have the same center. In particular, the following crucial ingredient for the proof of Theorem \ref{thm-main} can be deduced from the results in $\cite{Leh09a}$.

{\color{black}
\begin{oldtheorem}[Lehec]
\label{thm old A} 
Let $(\lambda_1,\dots,\lambda_n)$ be a system of coordinates and $\mu$ an even finite Borel measure on $\R^n$ such that $\mu(\R^n)>0$ and $\mu(H)=0$ holds for all hyperplanes $H\subseteq\R^n$.
Then there exists a Yao-Yao equipartition $\mathcal{P}$ of $\mu$ adapted to $(\lambda_1,\dots,\lambda_n)$ having its center at the origin. In particular, for any $u\in S^{n-1}$, there exists a Yao-Yao equipartition ${\cal P}$ of $\mu$, which is based on $u^\perp$ and has its center at the origin.   
\end{oldtheorem}
\begin{proof} 
For an even and integrable function $f:\R^n\to(0,\infty)$ and for a positive integer $k$, let $\mu_k:=\mu+(1/k)f\,dx$, and hence $\mu_k$ is an even finite Borel measure that assigns zero mass to any hyperplane and  $\textnormal{supp}\,\mu_k=\R^n$. 
According to \cite[Proposition 14]{Leh09a}, there exists a unique Yao-Yao equipartition ${\cal P}_k$ of $\mu_k$ which is adapted $(\lambda_1,\dots,\lambda_n)$.  Since $\mu_k$ is also even, uniqueness of the center ensures that for each $k$, the center of ${\cal P}_k$ is the origin (cf. \cite[Lemma 15]{Leh09a}). But then,  approximation (\cite[Lemma 12]{Leh09a}) shows that $\mu$ also has a Yao-Yao equipartition ${\cal P}$ with center $o$, adapted to $(\lambda_1,\dots,\lambda_n)$. 
In particular, taking $\lambda_1(x)=\langle u,x\rangle$,
${\cal P}$ is based on the hyperplane
$$
\{y\in \R^n:\,\lambda_1(y)=\lambda_1(o)\}=u^\perp,
$$
which is exactly what we wanted to show.
\end{proof}}

\section{Proof of Theorem~\ref{thm-Ball-conj} and Theorem \ref{thm-main}}
\label{secBall-conj}

\begin{proof}[{\color{black}Proof of Theorem~\ref{thm-main}}]
To prove \eqref{eq-thm-main}, 
we may clearly assume that $u=e_1$. Define the measure $\mu:={\bf 1}_Kx_1^2dx$. By Theorem \ref{thm old A},  there are two Yao-Yao partitions ${\cal P}_+$, ${\cal P}_-$ of $e_1^\perp$ and a vector $v\in  S^{n-1}$ with $\langle v,e_1\rangle>0$, such that the partition
$${\cal P}:=\{A'+ \R_{+} v:A'\in{\cal P}_+\}\cup\{A'+\R_{-} v:A'\in {\cal P}_-\}$$is a Yao-Yao equipartition of $\mu$ with center $o$. Let $T:\R^n\to\R^n$ be the linear map, such that $T(v)=\langle e_1,v\rangle e_1$ and  $Tz=z$ for $z\in e_1^\perp$, and hence $\det T=1$. Then, for an arbitrary $x\in \R^n$ we have
\begin{equation}\label{eq-T^{-1}}
\langle T^{-1}x,e_1\rangle=x_1\langle T^{-1}e_1,e_1\rangle+\langle T^{-1}(x-x_1e_1),e_1\rangle=x_1\langle v/\langle e_1,v\rangle, e_1\rangle=x_1.
\end{equation}
Set $\widetilde{K}:=TK$. Using \eqref{eq-T^{-1}} 
for any Borel set $\Omega$ in $\R^n$,
we get
$$
    \int_{(T\Omega)\cap \widetilde{K}}x_1^2\, dx= \int_{(T\Omega)\cap \widetilde{K}}\langle T^{-1}x,e_1\rangle^2\, dx=\int_{T^{-1}((T\Omega)\cap \widetilde{K})}\langle x,e_1\rangle^2\,dx=\int_{\Omega\cap K} x_1^2\,dx.
$$
In particular, we conclude that
\begin{equation}\label{eq-TA cap L}
     \int_{ \widetilde{K}}x_1^2\,dx=   \int_{K}x_1^2\,dx
\end{equation}
and if $\widetilde{\cal P}:=\{TA:A\in{\cal P}\}$, then $\widetilde{\cal P}$ is a Yao-Yao equipartition for $\tilde{\mu}:={\bf 1}_{\widetilde{K}}x_1^2dx$ based on $e_1^\perp$ as
\begin{equation}\label{eq-B cap L}
    \int_{B\cap \widetilde{K}}x_1^2\,dx=\frac{1}{2^n}\int_{\widetilde{K}}x_1^2\,dx,\qquad \forall B\in\widetilde{\cal P}.
\end{equation}
On the other hand, we have
\begin{equation}\label{v-vs-e1-in-polar-body}
\int_{{\widetilde{K}}^\circ}x_1^2\, dx=\int_{T^{-t}K^\circ}x_1^2\,dx
=\int_{K^\circ}\langle T^{-t}x,e_1\rangle^2\,dx
=\int_{K^\circ}\langle x,T^{-1}e_1\rangle^2\,dx
=\frac{1}{\langle e_1,v\rangle^2}\int_{K^\circ}\langle x,v\rangle^2\,dx.    
\end{equation}

Hence, $|\langle e_1,v\rangle|\leq 1$ and the isotropicity of $K^\circ$ (cf. \eqref{quasi-iso-u}) imply
\begin{equation}\label{eq-geq}
\int_{{\widetilde{K}}^\circ}x_1^2\, dx\geq \int_{K^\circ}\langle x,v\rangle^2\,dx=\int_{K^\circ}x_1^2\,dx.
\end{equation}
Now, \eqref{eq-TA cap L} 
and \eqref{eq-geq} yield
\begin{equation}\label{eq-8.5}\int_Kx_1^2\,dx\int_{K^\circ}x_1^2\,dx\leq \int_{\widetilde{K}}x_1^2\,dx\int_{{\widetilde{K}}^\circ}x_1^2\,dx.\end{equation}
Clearly, equality holds in \eqref{eq-8.5} if and only if $v=e_1$. It remains to prove that \eqref{eq-thm-main} holds with $u=e_1$ and with ${\widetilde{K}}$ in the place of $K$. To this end, observe that if $B\in \widetilde{\cal P}$, 
then there exist $C\subseteq e_1^\perp$ and $\varepsilon\in\{-1,1\}$, such that $B=C+\varepsilon\R_+ e_1$, where $C$ is of the form $\textnormal{pos}\,\{v_1,\dots,v_{n-1}\}$, for some linearly independent vectors $v_1,\dots,v_{n-1}\in e_1^\perp$. We wish to show that 
\begin{equation}\label{eq-wish}
\int_{B\cap {\widetilde{K}}}x_1^2\,dx\int_{B^*\cap {\widetilde{K}}^\circ}x_1^2\, dx\leq \frac{1}{2^{2n}}\left(\int_{B_2^n}x_1^2\,dx\right)^2,  
\end{equation}with equality if and only if for {\color{black}for this particular} $B\in \widetilde{\cal P}$, there exists a $o$-symmetric ellipsoid $E_B$, which is symmetric with respect to $e_1^\perp$ and satisfies $B\cap E_B={\color{black}B\cap {\widetilde{K}}}$.

By symmetry, we may assume that  
$\varepsilon=1$. Let $S:\R^n\to\R^n$ be any linear map (which depends on $B$) with $S(e_1)=e_1$, $S(C)=\R^n_+\cap e_1^\perp$ and $|\det S|=1$. Then, clearly,
$$S(B)=\R^n_+=(S(B))^*=S^{-t}(B^*).$$
Notice also that, for $x\in \R^n$, it holds
$\langle Sx,e_1\rangle=x_1$. 
Consequently, \eqref{eq-B cap L} gives
\begin{eqnarray}
\int_{\R^n_+\cap (S{\widetilde{K}})} x_1^2\,dx&=&\int_{S(B\cap {\widetilde{K}}) } x_1^2\,dx=\int_{B\cap {\widetilde{K}}}\langle Sx,e_1\rangle^2\,dx
=\int_{B\cap {\widetilde{K}}}x_1^2\,dx\label{eq-7}\\
&=&\frac{1}{2^n}\int_{\widetilde{K}}x_1^2\,dx.\label{eq-8}
\end{eqnarray}
In addition,
\begin{eqnarray}
\int_{\R^n_+\cap (S{\widetilde{K}})^{\circ}}x_1^2\,dx&=&\int_{S^{-t}(B^*\cap {\widetilde{K}}^\circ)}x_1^2\,dx=\int_{B^*\cap {\widetilde{K}}^\circ}\langle S^{-t}x,e_1\rangle^2=\int_{B^*\cap {\widetilde{K}}^\circ}\langle x,S^{-1}e_1\rangle^2\,dx\nonumber\\
&=&\int_{B^*\cap {\widetilde{K}}^\circ}x_1^2\,dx.
\label{eq-9}    
\end{eqnarray}
Now, \eqref{eq-wish} follows immediately from \eqref{eq-7}, \eqref{eq-9} and Theorem \ref{prop-Ball}. In fact, if equality holds in \eqref{eq-wish}, then equality necessarily holds in \eqref{eq-unc-ball} with $S{\widetilde{K}}$ in place of $K$. In other words, there exists an unconditional ellipsoid $E$, such that $\R^n_+\cap (S{\widetilde{K}})=\R^n_+\cap E$. Observe that the ellipsoid $E_B:=S^{-1}(E)$ is symmetric with respect to $e_1^\perp$ and that $B\cap E_B=B\cap {\widetilde{K}}$. This establishes the characterization of equality cases in \eqref{eq-wish}, claimed earlier.

Recall that the family $\big(\widetilde{\cal P}\big)^*=\{B^*:B\in\widetilde{\cal P}\}$ is also a partition of $\R^n$. Therefore, \eqref{eq-wish} and \eqref{eq-8} give
\begin{eqnarray}
\nonumber
 \int_{\widetilde{K}}x_1^2\,dx\int_{{\widetilde{K}}^\circ}x_1^2\,dx&=&\sum_{B\in\widetilde{\cal P}}\int_{{\widetilde{K}}\cap B}x_1^2\,dx\sum_{D\in(\widetilde{\cal P})^*}\int_{{\widetilde{K}}^\circ\cap D}x_1^2\,dx\\
 &=&\label{measure-decomp-into-2^n-parts}2^n\sum_{D\in(\widetilde{\cal P})^*}\int_{{\widetilde{K}}\cap D^*}x_1^2\,dx\int_{{\widetilde{K}}^\circ\cap D}x_1^2\,dx\\
 &\leq&\nonumber 2^n\sum_{D\in(\widetilde{\cal P})^*}\frac{1}{2^{2n}}\left(\int_{B_2^n}x_1^2\,dx\right)^2=\left(\int_{B_2^n}x_1^2\,dx\right)^2.
\end{eqnarray}
This completes the proof of \eqref{eq-thm-main}.\\

Finally, we assume that equality holds in \eqref{eq-thm-main}, and hence equality holds in \eqref{eq-8.5} and \eqref{eq-wish} for any $B\in \widetilde{\cal P}$, as well. We deduce from the equality conditions in \eqref{eq-8.5} that $v=e_1$, and hence ${\widetilde{K}}=K$ and $\widetilde{\cal P}=\mathcal{P}$. It follows from the equality conditions in \eqref{eq-wish} that for any $A\in \mathcal{P}$, there exists an $o$-symmetric ellipsoid $E_A\subseteq\R^n$ such that $A\cap K=A\cap E_A$. 
{\color{black}If $n=2$, $v=u=e_1$ implies $\mathcal{P}=\{\textnormal{pos}\{\varepsilon_1e_1,\varepsilon_2e_2\}:\varepsilon_i=\pm1\}$. By Theorem \ref{prop-Ball}, ${\widetilde{K}}$ is an unconditional ellipse, hence $K$ is an ellipse itself. The fact that $K$ is a Euclidean disk follows from the isotropicity of $K^\circ$.}

{\color{black}To handle the case $n\geq3$, we say that two cones $A,B\in \mathcal{P}$ are neighbors, if $A={\rm pos}\{A',e_1\}$ and $B={\rm pos}\{B',e_1\}$, where $(A'\cap B')\backslash {\rm pos(e_1)}\neq\emptyset$.
We claim that if $A,B\in \mathcal{P}$ are neighbors, then
\begin{equation}
\label{neighboring-ellipsoids}
E_A=E_B .
\end{equation}
Take another Yao-Yao equipartition $\mathcal{Q}$ based on $e_1^\perp$, with axis $e_1$, that contains some $\Omega \in\mathcal Q$ such that $\textnormal{int}(A\cap \Omega )\neq\emptyset$ and $\textnormal{int}(B\cap \Omega )\neq\emptyset$ (the existence of such an equipartition is due to the fact that we may arbitrarily choose a hyperplane $H$ of $e_1^\perp$ such that $\mathcal{P_+}$ is based on $H$; recall Theorem \ref{thm old A}) . Then there exists an ellipsoid $E_\Omega $ such that $\Omega \cap K=\Omega \cap E_\Omega $. 
But then $$\partial(A\cap E_A)\cap\textnormal{int}(A\cap \Omega )=\partial(\Omega \cap E_\Omega )\cap\textnormal{int}(A\cap \Omega ).$$
Since the boundary of an ellipsoid is a real analytic surface, we obtain $E_A=E_\Omega $. Similarly, $E_B=E_\Omega $ proving the claim \eqref{neighboring-ellipsoids}.}

Finally, it follows via induction on the dimension that  the graph $G$ whose vertices are the {\color{black}$2^{n-1}$ elements of $\{A\in\mathcal{P}:x_1\geq0~~\forall x\in A\}$ and the edges formed by the neighboring elements  is connected. Therefore,} for any $A,A'\in\mathcal{P}$, it holds $E_{A}=E_{A'}$ and hence $K$ is an origin symmetric ellipsoid. Since $K^\circ$ is isotropic, the ellipsoid $K$ is a Euclidean ball.
\end{proof}
\noindent{\bf Remark.}
    We would like to mention here that Yao-Yao partitions are not needed in dimensions 2 and 3 because one may simply equipartition $\mu$ into 4 (resp. 8) cones, using 2 lines (resp. 3 planes) passing through $o$, one of them being $u^\perp$. Indeed, if $n=2$ the existence of such an equipartition is trivial by the intermediate value theorem; whereas, if $n=3$, also by continuity, there exist two planes (one of them being $u^\perp$ and the other containing $u$) dividing $\mu$ into 4 equal parts and the existence of a third plane bisecting the measure of all 4 parts comes from appplying Ham-Sandwich Theorem on 2 adjacent parts and $B_2^n$ (recall symmetry).

\begin{proof}[Proof of Theorem \ref{thm-Ball-conj}]
Let $K$ be a symmetric convex body in $\Rn$ and let $K'$ be an isotropic linear image of $K$. 
Since the functional $$K\mapsto \int_K\int_{K^\circ}\langle x,y\rangle^2\, dy\, dx$$is invariant under non-singular linear maps, it suffices to prove \eqref{eq-conj} for $K'$ instead of $K$. 
By the isotropicity of $K'$, we get
$$\int_{K'}x_ix_j\, dx=0,\qquad i\neq j$$ and therefore,
\begin{equation}\label{eq-thm-conj-n=2-proof}\int_{K'}\int_{(K')^\circ}\langle x,y\rangle^2\, dy\, dx=\sum_{i=1}^n\int_{K'}x_i^2\, dx\int_{(K')^\circ}y_i^2\, dy.\end{equation}
The proof of \eqref{eq-conj} for $K'$ follows immediately from \eqref{eq-thm-conj-n=2-proof}, Theorem \ref{thm-main} and then again \eqref{eq-thm-conj-n=2-proof} with $K'=(K')^\circ=B_2^n$. Moreover, if equality holds in \eqref{eq-conj}, then equality must hold in \eqref{eq-thm-main} for $K'$ in the place of $K$ and for $u=e_1$. It follows by the equality cases in Theorem \ref{thm-main} that $K'$ is a Euclidean ball. That is, $K$ is an ellipsoid.
\end{proof}

\section{Weak stability of Theorem~\ref{thm-main}}
\label{secWeak-stab}

We note that if $L\subseteq\R^n$ is a symmetric convex body in isotropic position, then Kannan, Lov\'asz, Simonovits \cite{KLS95} prove that there exists an $o$-symmetric ball $B\subseteq L$ such that $L\subseteq nB$. If, in addition, $|L|=\omega_n$, then we deduce that
\begin{equation}
\label{ball-in-out-isotropic}
\frac1n\,B^n_2\subseteq L\subseteq nB^n_2\mbox{ \ and \ }\frac1n\,B^n_2\subseteq L^\circ\subseteq nB^n_2.
\end{equation}

Figalli, van Hintum, Tiba \cite{FHT25+}
 prove the following stability version of the Pr\'ekopa-Leindler inequality \eqref{PLnineq}.

\begin{theorem}[Figalli, van Hintum, Tiba]
\label{PLhstab}
For some explicit constant $\theta_n>1$ depending on $n\geq 1$, if  
$h,f,g:\,\R^n\to [0,\infty)$ are integrable with pos{\color{black}i}tive integral and $h$ is log-concave  such that
   $h(\frac12\,x+\frac12\,y)\geq \sqrt{f(x)g(y)}$ for
$x,y\in\mathbb{R}^n$, and
$$
\left(\int_{\R^n}  h\right)^2\leq (1+\varepsilon) \left(\int_{\R^n}f\right)\left( \int_{\R^n}g\right)
$$
for $\varepsilon\in(0,1]$, then
there exists $w\in\R^n$ such that for $a=\sqrt{\int_{\R^n}g/\int_{\R^n}f}$, we have
\begin{eqnarray*}
\int_{\R^n}|h(x)-a f(x-w)|\,dx&\leq &\theta_n\sqrt{\varepsilon}  
\cdot \int_{\R^n}h \\
\int_{\R^n}|h(x)-a^{-1}g(x+w)|\,dx&\leq &\theta_n\sqrt{\varepsilon}   
\cdot\int_{\R^n}h.
\end{eqnarray*}
\end{theorem}

For $t>0$ and $u\in S^{n-1}$, let us consider the strip
$$
\Theta_{u,t}=\{(x_1,\ldots,x_n)\in\R^n:|\langle x,u\rangle|\leq t\}.
$$
The key statement is the following stability version of Theorem~\ref{prop-Ball}.

\begin{proposition}
\label{coordinatewisestab0}
For measurable $X,Y\subseteq \R^n_+$ with $|X|>0$ and $|Y|>0$, and $i=1,\ldots,n$, if $\langle x,y\rangle\leq 1$ for any $x\in X$ and $y\in Y$, and
\begin{equation}
\label{Ballineq-equa-eps}
\int_Xx_i^2\;dx \cdot \int_Y x_i^2\;dx\geq (1-\varepsilon)\left(\int_{\R^n_+\cap B^n_2}x_i^2\;dx\right)^2
\end{equation}
for some $\varepsilon\in(0,\frac12)$, then
there exists  positive definite diagonal matrix $\Phi$ such that 
\begin{align}
\label{coordinatewisestab-strip-eq}
\left|\big((\R^n_+\cap B^n_2)\Delta (\Phi X)\big)\left\backslash \Theta_{e_i,\Lambda}\right.\right|<&\frac{\theta_n}{\Lambda^2}\cdot \sqrt{\varepsilon} 
\mbox{ \ and }
&\left|\big((\R^n_+\cap B^n_2)\Delta (\Phi^{-1} Y)\big)\left\backslash \Theta_{e_i,\Lambda}\right.\right|<&\frac{\theta_n}{\Lambda^2}\cdot\sqrt{\varepsilon}
\end{align}
for any $\Lambda>0$ where the explicit $\theta_n>1$ depends on $n$. Furthermore, if $X,Y\subseteq RB_2^n$ for some $R>0$, then,
for $\varepsilon$ small enough (depending only on $n$), if $t_i$, $i=1,\dots,n$ are the eigenvalues of $\Phi^{-1}$ then it holds $t_i\leq2R$ and $t_i^{-1}\leq2R$, for all $i$. 
\end{proposition}
\begin{proof} 
We may assume that $i=1$, and after rescaling $X$ and $Y$, we may assume that
$$
\int_{X}{\color{black}x_1^2}\;dx=\int_{Y}x_1^2\;dx,
$$
and hence, using also Theorem~\ref{thm-main}, we have
$$
(1-\varepsilon)\int_{\R^n_+\cap B^n_2}x_1^2\;dx<\int_{X}x_1^2\;dx\leq \int_{\R^n_+\cap B^n_2}x_1^2\;dx.
$$
We consider again the functions {\color{black}$f,g,h$ defined in \eqref{f-function}, \eqref{g-function} and \eqref{h-function}, which satisfy \eqref{integrals-substitut} as well as
$h(\frac12\,x+\frac12\,y)\geq \sqrt{f(x)g(y)}$} for any
$x,y\in\mathbb{R}^n$. Therefore, combining \eqref{Ballineq-equa-eps} and the stability version Theorem~\ref{PLhstab} of the Pr\'ekopa-Leindler inequality yields the existence of a $w\in \R^n$ such that
\begin{align*}
\int_{\R^n}|h(x)-f(x-w)|\,dx\leq &\tilde{\theta}_n\sqrt{\varepsilon}  
\cdot \int_{\R^n}h \\
\int_{\R^n}|h(x)-g(x+w)|\,dx\leq &\tilde{\theta}_n\sqrt{\varepsilon}   
\cdot\int_{\R^n}h
\end{align*}
where $\tilde{\theta}_n>1$ depends on $n$. For $C_n=\tilde{\theta}_n\int_{\R^n_+\cap B^n_2}x_1^2\;dx$ and $b_j=e^{w_j}$, $j=1,\ldots,n$, we deduce that
\begin{align*}
\int_{(\R^n_+\cap B^n_2)\Delta \Phi X}x_1^2\;dx\leq & C_n\sqrt{\varepsilon}\\
\int_{(\R^n_+\cap B^n_2)\Delta \Phi^{-1}Y}x_1^2\;dx\leq & C_n\sqrt{\varepsilon}
\end{align*}
where $\Phi$ is the diagonal transformation with eigenvalues $b_1,b_2,\ldots,b_n$. In turn, we conclude
\eqref{coordinatewisestab-strip-eq} by $x_1^2\geq \Lambda^2$ for  $x\not\in \Theta_{1,\Lambda}$. 

Regarding the bound for the eigenvalues, 
assume that $t_j>2R$ for some $j$. But then $\Phi X\subseteq \Theta_{e_j,\frac{1}{2}}\subseteq \bigcup_{i=1}^n\Theta_{e_i,\frac{1}{2}}$, hence
$$\big((\R^n_+\cap B^n_2)\Delta(\Phi X)\big)\backslash\Theta_{e_j,\frac{1}{2}}\supseteq \big((\R^n_+\cap B^n_2)\Delta(\Phi X)\big)\backslash\bigcup_{i=1}^n\Theta_{e_i,\frac{1}{2}}=(\R^n_+\cap B^n_2)\backslash\bigcup_{i=1}^n\Theta_{e_i,\frac{1}{2}},$$   
which contradicts \eqref{coordinatewisestab-strip-eq} with $\Lambda=\frac{1}{2}$ if $\varepsilon$ is small enough. The bound for $t_i^{-1}$ is similar.
\end{proof}

\begin{proposition}\label{prop-op}
Let $R,\delta>0$ and ${\cal P}$ be a Yao-Yao partition of $\R^n$, such that for all $A\in\mathcal{P}$ it holds $|(RB_2^n)\cap A|>\delta$. Assume that for some (linearly independent) unit vectors $v_1,\dots,v_n$, it holds ${\rm pos}\{v_1,\dots,v_n\}\in{\cal P}$ and set $\Psi:\R^n\to\R^n$ to be the linear map given by $\Psi(v_i)=e_i$, $i=1,\dots,n$. Then, there exists a constant $C=C(n,R,\delta)>0$ that depends only on $n,R,\delta$, such that $$\|\Psi\|_{\textnormal{op}},\|\Psi^{-1}\|_{\textnormal{op}}\leq C.$$ 
\end{proposition}
\begin{proof}
 First notice that since $\{e_1,\dots,e_n\}$ is an orthonormal basis, the estimate for $\|\Psi^{-1}\|_{\textnormal{op}}$ is automatic. Indeed, for $x=x_1v_1+\dots x_nv_n\in\R^n$, we have
 $$\left\|\sum_{i=1}^nx_iv_i\right\|^2=\sum_{i,j=1}^nx_ix_j\langle v_i,v_j\rangle\leq \sum_{i,j=1}^n|x_i||x_j|=(|x_1|+\dots+|x_n|)^2\leq n(x_1^2+\dots+x_n^2)=n\|\Psi(x)\|^2,$$
thus $\|\Psi^{-1}\|_{\textnormal{op}}\leq \sqrt{n}$. To deduce the estimate for $\|\Psi\|_{\textnormal{op}}$, we may clearly assume that $R=1$. We prove this estimate (with $R=1$) by induction in $n$, with the case $n=1$ being trivial. Let us assume that the estimate $\|\Psi\|_{\textnormal{op}}\leq C(n-1,\delta)$ is valid in $\R^{n-1}$.

If ${\cal P}$, $v_1,\dots,v_n$ are as in the statement of Proposition  \ref{prop-op}, then we may clearly assume that ${\cal P}$ is based on $e_n^\perp\equiv \R^{n-1}$, that $v_n$ is the axis of ${\cal P}$ and that $\langle v_n,e_n\rangle>0$. Let ${\cal P}_+$ be the Yao-Yao partition of $e_n^\perp$ induced by $\mathcal{P}$, $v_n$. Then, ${\rm pos}\{v_1,\dots,v_{n-1}\}\in{\cal P}_+$. Set also $v'_n:=v_n-\langle v_n,e_n\rangle e_n$ to be the orthogonal projection of $v_n$ onto $e_n^\perp$. By approximation, we may assume that $v'_n\neq o$. Clearly, there exists a (computable) constant $0<s_0<1$, such that \begin{equation}\label{eq-ctr}|B_2^n\cap \Theta_{e_n,s_0}|<\delta.\end{equation}Set $H^+:=\{x\in e_n^\perp:\langle x,v_n'\rangle\geq 0\}$. Then, \eqref{Yao-Yao-cone-in-halfspace} shows that there exists $A'\in{\cal P}_+$, such that $A'\subseteq H^+$. On the other hand, it can be easily seen that
$$B_2^n\cap {\rm pos}(H^+\cup \{v_n\})\subseteq B_2^n\cap \{x:\in \R^n:0\leq \langle x,v_n\rangle\leq \langle v_n,e_n\rangle\}\subseteq B_2^n\cap \Theta_{e_n,\langle v_n,e_n\rangle}.$$We conclude that for $A:={\rm pos}(A'\cup\{v_n\})\in{\cal P}$, 
$$|A\cap B_2^n|\leq |B_2^n\cap \Theta_{e_n,\langle v_n,e_n\rangle}|$$and therefore by \eqref{eq-ctr}, we arrive at
$$\langle v_n,e_n\rangle>s_0.$$Consider the linear map $W:\R^n\to\R^n$ given by $W|e_n^\perp=id_{e_n^\perp}$, $W(v_n)=e_n$. Then, one can easily compute $$\|W\|_{\textnormal{op}}\leq \left(1-\sqrt{1-s_0^2}\right)^{-1/2}=:t_0,
\qquad |\det W|\geq 1.$$For any $A'\in {\cal P}_+$, we have $A:={\rm pos}(A'\cup \{v_n\})\in {\cal P}$ and also,
\begin{eqnarray*}
\delta&\leq& |\det W||A\cap B_2^n|=|W(A\cap B_2^n)|\leq |W(A)\cap (t_0B_2^n)|=t_0^n|W(A)\cap B_2^n|\\
&=&t_0^n\int_0^1\big|B_2^n\cap {\rm pos}(A'\cup\{e_n\})\cap (e_n^\perp+te_n)\big|_{n-1}dt\\
&=&t_0^n\int_0^1\big|\sqrt{1-t^2}(B_2^n\cap {\rm pos}(A'\cup\{e_n\})\cap e_n^\perp)\big|_{n-1}dt\\
&=&t_0^n\int_0^1(1-t^2)^{(n-1)/2}dt\cdot\big|B_2^{n-1}\cap A'|_{n-1}=:a_n\big|B_2^{n-1}\cap A'\big|_{n-1}.
\end{eqnarray*}
It follows that the assumptions of Proposition \ref{prop-op} are satisfied for the Yao-Yao partition ${\cal P}_+$ of $\R^{n-1}$ with $R=1$ and $\delta':=a_n^{-1}\delta$ in the place of $\delta$, hence the inductive hypothesis gives
$$\|\Psi_1\|_{\textnormal{op}}<C(n-1,\delta'),$$where $\Psi_1:\R^n\to\R^n$ is the linear map given by $\Psi_1(v_i)=e_i$, $i=1,\dots,n-1$, $\Psi_1(e_n)=e_n$. We conclude that since $\Psi=\Psi_1\circ W$,
$$\|\Psi\|_{\textnormal{op}}\leq \|\Psi_1\|_{\textnormal{op}}\|W\|_{\textnormal{op}}\leq C(n-1,\delta')t_0,$$proving our claim.
\end{proof}

Now, we are ready to prove a rather technical stability version of 
Theorem~\ref{thm-main} that will lead to a stability version of Theorem~\ref{thm-Ball-conj}. 

\begin{proposition}
\label{thm-main-stab}
Let $u\in S^{n-1}$ and  $K\subseteq \R^n$ be a symmetric convex body, 
such that $K^\circ$ is in isotropic position with $|K^\circ|=\omega_n$, 
 and let $\mathcal{P}$ be a Yao-Yao equipartition of the measure  $d\mu(x)=\langle x,u\rangle^2\mathbf{1}_K(x)\,dx$ where $\mathcal{P}$ is centered at the origin, is based on  $u^\bot$ and has axis $v\in S^{n-1}$. There exist explicit $C>0$, $R>n$ and $\varepsilon_0\in(0,1)$ depending only on $n$ satisfying that if
\begin{equation}
\label{K-sqxu-eps-stab-eq}
\int_K\langle x,u\rangle^2\,dx\cdot\int_{K^\circ}\langle y,u\rangle^2\,dy\geq (1-\varepsilon)\left(\int_{B^n_2}\langle x,u\rangle^2\,dx\right)^2
\end{equation}
for $\varepsilon\in(0,\varepsilon_0)$, then 
\begin{equation}
\label{thm-main-stabuv-eq}
|\langle u,v\rangle|\geq {\color{black}\sqrt{1-\varepsilon}}
\end{equation}
and for any $A\in\mathcal{P}$, one finds an $o$-symmetric ellipsoid $E_A$ (depending on $A$ and $K$) such that
\begin{align}
\label{thm-main-stab-nostrip}
\big|\big((K\Delta E_A)\cap A\big)\backslash \Theta_{u,\Lambda}\big|\leq &\frac{C}{\Lambda^2}\cdot\varepsilon^{\frac{1}{2}} \mbox{ \ 
 for any $\Lambda>0$,}\\
\label{thm-main-stab-in-out-radius}
\frac1R B^n_2\subseteq& E_A\subseteq RB^n_2. 
 \end{align}
\end{proposition}
\begin{proof} As $K^\circ$ is in isotropic position with $|K^\circ|=\omega_n$, $K$ satisfies \eqref{ball-in-out-isotropic}.
We may assume that $\langle u,v\rangle>0$. 
{ \color{black}Let us consider again the linear transfrom
$T|_{u^\bot}=\textnormal{Id}$,  $T(v)=\langle u,v\rangle u$, which has the properties }
 $\det T=1$ and $T^{-1}(u)=\frac{v}{\langle u,v\rangle}$ and the body $\widetilde{K}=TK$.
{\color{black}Recalling \eqref{eq-TA cap L} and \eqref{v-vs-e1-in-polar-body} one has}
\begin{equation}
\label{sqxuKcirc-sqxuTKcirc}
\int_{\widetilde{K}}\langle y,u\rangle^2\,dy\cdot \int_{\widetilde{K}^\circ}\langle y,u\rangle^2\,dy=\frac1{\langle u,v\rangle{\color{black}^2}}\int_K\langle x,u\rangle^2\,dx\cdot\int_{K^\circ}\langle x,v\rangle^2\,dx.
\end{equation}
As the left hand side of \eqref{sqxuKcirc-sqxuTKcirc} is at most $\left(\int_{B^n_2}\langle x,u\rangle^2\,dx\right)^2$ by Theorem~~\ref{thm-main}, we deduce the estimate $\langle u,v\rangle>{\color{black}\sqrt{1-\varepsilon}}$ in \eqref{thm-main-stabuv-eq} from \eqref{sqxuKcirc-sqxuTKcirc} and the condition \eqref{K-sqxu-eps-stab-eq}. Further, for small enough $\varepsilon$ this yields \begin{equation}\label{thm13-stab-Toperator-norm-claim}
{\color{black}\|T^{-1}\|_{\textnormal{op}},~\|T\|_{\textnormal{op}}\leq 2.}
\end{equation}

To prove 
\eqref{thm-main-stab-nostrip}, we observe that (by \eqref{eq-B cap L}) $\widetilde{\mathcal{P}}:=\{TA:\,A\in\mathcal{P}\}$ is a Yao-Yao equipartition of $\R^n$ with respect to the measure $d\tilde{\mu}(y)=\langle y,u\rangle^2\mathbf{1}_{\widetilde{K}}(y)\,dy$; moreover, $\widetilde{\mathcal{P}}$ is based on $u^\bot$ and has $u$ as its axis and the origin as its center. 
Let $A\in\mathcal{P}$. 
 We may assume that $A=A'+\R_+v$ where $A'\in\mathcal{P}_+$ and $\mathcal{P}_+$ is the induced by $\mathcal{P}$, $u$ Yao-Yao partition of $u^\bot$, and $A'={\rm pos}\{v_1,\ldots,v_{n-1}\}$ for independent $v_1,\ldots,v_{n-1}\in u^\bot\cap S^{n-1}$.
In particular, $\widetilde{A}:=TA={\rm pos}\{A',u\}\in \widetilde{\mathcal{P}}$.

It follows from the condition \eqref{K-sqxu-eps-stab-eq}, \eqref{measure-decomp-into-2^n-parts} and \eqref{sqxuKcirc-sqxuTKcirc} that
\begin{align*}
(1-\varepsilon)2^{2n}\left( \int_{\R^n_+\cap B^n_2}\langle x,u\rangle^2\,dx\right)^2=&
(1-\varepsilon)\left(\int_{B^n_2}\langle x,u\rangle^2\,dx\right)^2\leq\int_{\widetilde{K}^\circ}\langle x,u\rangle^2\,dx\cdot \int_{\widetilde{K}}\langle x,u\rangle^2\,dx\\
=&
2^n\sum_{D\in\widetilde{\mathcal{P}}}\int_{D\cap \widetilde{K}}\langle x,u\rangle^2\,dx\cdot \int_{D^*\cap\widetilde{K}^\circ}\langle x,u\rangle^2\,dx.
\end{align*}  
But, for any $D\in \widetilde{\mathcal{P}}$, 
 recalling \eqref{eq-wish}, 
we deduce that
\begin{equation}
\label{tildeAcond}
\int_{\widetilde{A}\cap \widetilde{K}}\langle x,u\rangle^2\,dx\cdot \int_{\widetilde{A}^*\cap\widetilde{K}^\circ}\langle x,u\rangle^2\,dx\geq 
(1-2^{2n}\varepsilon)\left(\int_{\R^n_+\cap B^n_2}\langle x,u\rangle^2\,dx\right)^2.
\end{equation}

{\color{black}
{\color{black}Next, notice that $V:=|\widetilde{A}\cap B_2^n|$ has an explicit lower bound that depends only on $n$.
Indeed, from \eqref{ball-in-out-isotropic}, \eqref{thm13-stab-Toperator-norm-claim} and the facts that $\widetilde{\mathcal{P}}$ is a Yao-Yao equipartition for $\widetilde{\mu}$
and  $\langle u,x\rangle^2\leq (2n)^2$ for $x\in \widetilde{K}$, we obtain
\begin{equation}\label{V-lower-bound}V\geq\frac{|\widetilde{A}\cap \widetilde{K}|}{(2n)^n}\geq \frac1{(2n)^{2+n}}\int_{\widetilde{A}\cap \widetilde{K}}\langle u,x\rangle^2\,dx =\frac{1}{2^{2(1+n)}n^{2+n}}\int_{\widetilde{K}}\langle u,x\rangle^2\,dx\geq \frac{1}{2^{2(1+n)}n^{2+n}}\int_{\frac1{2n}B^n_2}\langle u,x\rangle^2\,dx   
\end{equation}}
Without loss of generality, we choose $u=e_n$. We consider the linear transform $S$ with  {\color{black}$Sv_i=e_i$} for $i=1,\ldots,n-1$ and $Su=u$. 
{\color{black}By Proposition \ref{prop-op}
 \begin{equation}\label{thm13-stab-Soperator-norm-claim}\|S\|_{\textnormal{op}},~\|S^{-1}\|_{\textnormal{op}}\leq M\end{equation}}
 for some constant $M>1$ that depends only on $n$. 
 
Let $L:=S\widetilde{K}$.
Recalling \eqref{eq-7} and \eqref{eq-9}, \eqref{tildeAcond} implies}
\begin{align}
\int_{\R^n_+\cap L}\langle x,u\rangle^2\,dx\cdot \int_{\R^n_+\cap L^\circ}\langle x,u\rangle^2\,dx
\label{StildeAcond}
\geq &(1-2^{2n}\varepsilon)\left(\int_{\R^n_+\cap B^n_2}\langle x,u\rangle^2\,dx\right)^2.
\end{align}

Given \eqref{StildeAcond}, we deduce from 
Proposition~\ref{coordinatewisestab0} 
the existence of a $\tilde{\theta}_n>1$ depending on $n$ and a positive definite diagonal matrix $\Phi$ (with respect to the orthonormal basis $e_1,\ldots,e_n$) depending on $L$ and $\varepsilon$ such that
\begin{align}
\label{coordinatewisestab-strip-PhiL}
\left|\big((\R^n_+\cap B^n_2)\Delta\Phi (\R^n_+\cap L)\big)\left\backslash \Theta_{u,\widetilde{\Lambda}}\right.\right|<&\frac{\tilde{\theta}_n}{\widetilde{\Lambda}^2}\cdot \sqrt{\varepsilon} 
\end{align}
for any $\widetilde{\Lambda}>0$ (recall that $e_n=u$).
Furthermore, by \eqref{ball-in-out-isotropic}, \eqref{thm13-stab-Toperator-norm-claim}, \eqref{thm13-stab-Soperator-norm-claim} and Proposition \ref{coordinatewisestab0}, the eigenvalues of both $\Phi$ and $\Phi^{-1}$ are bounded by $4nM$.
In particular, $\det\Phi^{-1}\leq (4nM)^n$ and $\Phi^{-1}\Theta_{u,\widetilde{\Lambda}}\subseteq\Theta_{u,\Lambda}$ for $\Lambda=4nM\widetilde{\Lambda}$.

It follows from  this and from 
\eqref{coordinatewisestab-strip-PhiL} 
 that the unconditional ellipsoid $\widetilde{E}=\Phi^{-1} B^n_2$ satisfies
 
$$\left|\big((L\Delta\widetilde{E})\cap\R^n_+\big)\left\backslash \Theta_{u,\Lambda}\right.\right|<\frac{C}{\Lambda^2}\cdot \sqrt{\varepsilon} 
$$
for any $\Lambda>0$, where $C=(4nM)^n\tilde{\theta}_n$ depends only on $n$. Now $\Psi:=(ST)^{-1}$ satisfies {\color{black}$\|\Psi\|_{\textnormal{op}}\leq2M$}, $\Psi \Theta_{u,\Lambda}=\Theta_{u,\Lambda}$ (this follows from \eqref{eq-T^{-1}}) and $A\cap K=\Psi(\R^n_+\cap L)$. Thus, we conclude 
\eqref{thm-main-stab-nostrip} with $E_A=\Psi\widetilde{E}$.

{\color{black}Finally, \eqref{thm-main-stab-in-out-radius} is deduced from $\|\Phi\|_{\textnormal{op}},~\|\Phi^{-1}\|_{\textnormal{op}}\leq 4nM$ and $\|\Psi\|_{\textnormal{op}},~\|\Psi^{-1}\|_\textnormal{op}\leq2M$.}
\end{proof}

\section{Stability of Theorem~\ref{thm-Ball-conj}}
\label{secBall-ineq-stab}


During the proof of our stability statements concerning a symmetric convex body $K\subseteq\R^n$, we use various $o$-symmetric ellipsoids to approximate $K$ on a smaller part of $\R^n$, and the following lemma ensures that these ellipsoids are very close to each other.

\begin{lemma}
\label{EEprime-sandwiched}
For $n\geq 2$, $R>1$ and  $s\in (0,\frac12)$, there exist explicit $\delta,C>0$ depending on $n,R,s$ such that if $\sigma\subseteq \R^n$ is a convex cone with $r(\sigma\cap B^n_2)\geq s$, and the $o$-symmetric ellipsoids $E,E'\subseteq\R^n$ satisfy 
$$\frac1{R}\,B^n_2\subseteq E,
{\color{black}E'}\subseteq RB^n_2 \mbox{ \ and \ }|(E\Delta E')\cap \sigma|\leq \delta,
$$ 
then
$$
|E\Delta E'|\leq C \big|(E\Delta E')\cap \sigma\big|.
$$
\end{lemma}
\begin{proof} During the proof of Lemma~\ref{EEprime-sandwiched}, we write $g\gg h$ (resp. $h\ll g$) 
for some positive expressions $g$ and $h$ if $g\geq ch$ (resp. $g\leq ch$) where $c>0$ is an explicit constant that depends only on $n,R,s$. 

There exists a positive definite symmetric matrix $\Phi\in{\rm GL}(n)$ such that $\Phi E'=B^n_2$, and the eigenvalues of $\Phi$ are between $\frac1{R}$ and $R$, and hence $R^{-n}\leq \det\Phi\leq R^n$. In particular, 
the convex cone $\tilde{\sigma}=\Phi \sigma$ satisfies 
\begin{align}
\label{rtildeEB-tildesigma}
r(\tilde{\sigma}\cap B^n_2)\geq & \frac{s}{R},
\end{align}
and for the $o$-symmetric ellipsoid $\widetilde{E}=\Phi E$ it is sufficient to find explicit $\delta,C>0$ depending on $n,R,s$ such that 
\begin{equation}
\label{tildeEB-sandwiched-claim}
\mbox{if \ }\big|(\widetilde{E}\Delta B^n_2)\cap \tilde{\sigma}\big|\leq \delta\cdot R^n, \mbox{ \ then }
|\widetilde{E}\Delta B^n_2|\leq \frac{C}{R^n} \cdot\big|(\widetilde{E}\Delta B^n_2)\cap {\color{black}\widetilde{\sigma}}\big|.
\end{equation}

 To prove \eqref{tildeEB-sandwiched-claim}, 
 assume without loss of generality that the radial function of $\tilde{E}$ is 
 $$\varrho_{\widetilde{E}}(x)=\frac{1}{\sqrt{\sum_{i=1}^n(a_ix_i)^2}},$$ where $a_1\geq\dots\geq a_n>0$
 . The core claim is that if 
 \eqref{rtildeEB-tildesigma} holds, and $\max_{i=1}^n\{|a_i-1|\}\geq \varepsilon$ for $\varepsilon\in(0,\frac12)$, then
\begin{equation}
\label{E-cap-Bn2-maxai}
\left|(\widetilde{E}\Delta B^n_2)\cap \tilde{\sigma}\right|\geq c\varepsilon
\end{equation}
where 
$c>0$
depends only on $n,R,s$.

For $u\in S^{n-1}$ and $\psi\in(0,\frac{\pi}2]$, we write $B(u,\psi)$ 
to denote the spherical ball (cap) of center $u$ and geodesic radius $\psi$. 
According to  \eqref{rtildeEB-tildesigma}, there exists $w\in S^{n-1}$ such that 
\begin{equation*}
B(w,\alpha)\subseteq \tilde{\sigma}\cap S^{n-1} \mbox{ \ where   }\sin\alpha=\frac{s}{R}.
\end{equation*}
Without loss of generality $\alpha\leq\frac{\pi}{4}$. For appropriate $p=\sum_{i=1}^np_ie_i\in B(w,\alpha)$ having geodesic distance $\alpha/2$ from $w$ it is true that any $x\in B(p,\frac{\alpha}{4})\subseteq \widetilde{\sigma}\cap S^{n-1}$ has distance at least $\alpha/4$ from $e_1,\dots,e_n$, hence, by $R^{-2}B_2^n\subseteq\widetilde{E}\subseteq R^2B_2^n$,
\begin{equation}\label{lower bound for rho_E-1}
    |1-\varrho_{\widetilde{E}}(x)|\gg\left|\frac{1}{\varrho_{\widetilde{E}}(x)^2}-1\right|=\left|\sum_{i=1}^nb_ix_i^2\right|
\end{equation}
where $b_i:=a_i^2-1$ and $b_1\geq\dots\geq b_n$. We focus on the case 
$a_n\leq1-\varepsilon$, thus $b_n\leq-2\varepsilon+\varepsilon^2<-\varepsilon$. The case $a_n>1-\varepsilon$ (and hence $a_1\geq1+\varepsilon$) is similar.
Call $k\in\{1,\dots,n\}$ the smallest integer for which $b_k<0$,
thus $b_i<0$ for $i\geq k$ and $b_i\geq0$ for $i<k$. 
Let $\epsilon_0\in\{\pm1\}$ be such that $\epsilon_0\sum_{i=1}^nb_ip_i^2\geq0$. For an appropriate constant $c_0>0$ depending only on $n,\alpha$ the set
$$\Xi:=\left\{x\in B\left(w',\frac{\alpha}{4}\right):\epsilon_0x_i^2\geq p_i^2+c_0 \textnormal{ for } i<k \textnormal{ and } \epsilon_0x_i^2\leq p_i^2-c_0 \textnormal{ for } i\geq k\right\}$$ has positive $\Ha$ measure, bounded from below by a positive constant that depends only on $n,\alpha$.
Furthermore, for any $x\in\Xi$,
$$\epsilon_0\sum_{i=1}^nb_ix_i^2\geq \epsilon_0\sum_{i=1}^{k-1}b_i(p_i^2+c_0)+\epsilon_0\sum_{i=k}^nb_i(p_i^2-c_0)\geq \epsilon_0c_0\left(\sum_{i=1}^{k-1}b_i-\sum_{i=k}^nb_i\right)\geq -\epsilon_0c_0b_n\geq\epsilon_0c_0\varepsilon,$$
which, combined with \eqref{lower bound for rho_E-1}, gives
$|1-\varrho_{\widetilde{E}}(x)|\gg \varepsilon$ and hence
$$\left|(\widetilde{E}\Delta B^n_2)\cap \tilde{\sigma}\right|\geq\frac1n\int_{B(w',\frac{\alpha}{4})}\left|1-\varrho_{\widetilde{E}}^n\right|\,d\mathcal{H}^{n-1} 
\gg \int_{\Xi}\left|1-\varrho_{\widetilde{E}}^n\right|\,d\mathcal{H}^{n-1}
\geq \int_{\Xi}\left|1-\varrho_{\widetilde{E}}\right|\,d\mathcal{H}^{n-1}\gg
\varepsilon
$$



Finally, having the core claim \eqref{E-cap-Bn2-maxai} at hand, we set $\delta=\frac{c}{2}$, where $c$ is the constant in \eqref{E-cap-Bn2-maxai}, and $\varepsilon=\max_{i=1}^n\{|a_i-1|\}$. If $\left|({\color{black}\widetilde{E}}\Delta B^n_2)\cap \tilde{\sigma}\right|\leq \delta$, then  $\varepsilon\leq \frac1{2}$, and hence 
$$
\left|{\color{black}\widetilde{E}}\Delta B^n_2\right|\leq \omega_n\left(\frac{1}{(1-\varepsilon)^n}-\frac{1}{(1+\varepsilon)^n}\right)\ll((1+\varepsilon)^n-(1-\varepsilon)^n)\ll\varepsilon
\leq \frac{1}{c}\cdot \left|({\color{black}\widetilde{E}}\Delta B^n_2)\cap \tilde{\sigma}\right|,
$$
verifying \eqref{tildeEB-sandwiched-claim} 
and in turn proving Lemma~\ref{EEprime-sandwiched}.
\end{proof}

We are ready to verify the stability version Theorem~\ref{Ball-ineq-stab}, - equivalent to Theorem~\ref{Ball-ineq-stab0} - of Theorem~\ref{thm-Ball-conj}.
\begin{theorem}
\label{Ball-ineq-stab}
There exist $C_n>1$ and $\varepsilon_0\in(0,1)$ depending on $n\geq 2$ such that if $\varepsilon\in (0,{\color{black}\varepsilon_0})$ and symmetric convex body $K\subseteq \R^n$ satisfy
\begin{equation}
\label{Ball-ineq-uncond-stab-cond}
\int_K\int_{K^\circ}\langle x,y\rangle^2\,dydx> \left(1-\varepsilon\right) \int_{B^n_2}\int_{B^n_2}\langle x,y\rangle^2\,dydx,
\end{equation}
then 
$$|E\Delta K|\leq C_n\sqrt{\varepsilon}$$
for some $o$-symmetric ellipsoid $E\subseteq \R^n$.
\end{theorem}
\begin{proof}
We may assume that $K^\circ$ is in isotropic position and $|K^\circ|=\omega_n$, and hence \eqref{ball-in-out-isotropic} yields that
\begin{equation}
\label{thm12-K-ball-in-out-isotropic}
\frac1n\,B^n_2\subseteq K\subseteq nB^n_2.
\end{equation}

Let us discuss some consequences of the conditions \eqref{Ball-ineq-uncond-stab-cond} and that $K^\circ$ is in isotropic position before we start the actual proof of Theorem~\ref{Ball-ineq-stab}. For any $u\in S^{n-1}$, we claim that
\begin{equation}
\label{Kxu-neps-cond}
\int_K\langle x,u\rangle^2\,dx\cdot\int_{K^\circ}\langle y,u\rangle^2\,dy\geq (1-n\varepsilon)\left(\int_{B^n_2}\langle x,u\rangle^2\,dx\right)^2.
\end{equation}
To prove \eqref{Kxu-neps-cond}, 
{\color{black}we may assume $u=e_n$.} Recalling the proof \eqref{eq-thm-conj-n=2-proof}, we deduce from {\color{black} \eqref{Ball-ineq-uncond-stab-cond}} that
{\color{black}\begin{align*}
\sum_{i=1}^n\int_K x_i^2\,dx\cdot\int_{K^\circ} y_i^2\,dy
\geq (1-\varepsilon)\int_{B^n_2}\int_{B^n_2}\langle x,y\rangle^2\,dydx=(1-\varepsilon)n\left(\int_{B^n_2}\langle x,u\rangle^2\,dx\right)^2.
\end{align*}}
As $\int_K x_i^2\,dx\cdot\int_{K^\circ}y_i^2\,dy\leq \left(\int_{B^n_2}\langle x,u\rangle^2\,dx\right)^2$ holds for $i=1,\ldots,n$ by Theorem~\ref{thm-main}, we conclude \eqref{Kxu-neps-cond}.

Next, we discuss some properties of Yao-Yao partitions.
For any $u\in S^{n-1}$, we consider the even measure $d\mu(x)=\langle x,u\rangle^2\mathbf{1}_K(x)\,dx$,  and a Yao-Yao equipartition $\mathcal{P}$ of $\R^n$ for $\mu$ based on $u^\bot$, centered at the origin and having an axis $v\in S^{n-1}$ satisfying that $\langle u,v\rangle>0$. It follows from \eqref{Kxu-neps-cond} and
Proposition~\ref{thm-main-stab} that there exist $\widetilde{C}>0$, $R>n$ and $\tilde{\varepsilon}_0\in(0,1)$ depending only on $n$  such that if $\varepsilon<\tilde{\varepsilon}_0$, then 
\begin{equation}\label{thm12-uv-close-dot-prod}
\langle u,v\rangle\geq \sqrt{1-n\varepsilon};    
\end{equation}
moreover, for any
$A=\textnormal{pos}(v_1,\dots,v_n)\in\mathcal{P}$, $v_1,\dots,v_n\in S^{n-1}$, one finds  an $o$-symmetric ellipsoid $E_{A}$  (depending on $A$ and $K$), satisfying \eqref{thm-main-stab-nostrip} and \eqref{thm-main-stab-in-out-radius}.

In addition, we claim that there exists an $\tilde{r}>0$ depending only on $n$ such that if $A\in\mathcal{P}$, then
\begin{equation}
\label{inradius-cone-inYaoYao}
r(A\cap K)\geq \tilde{r}\mbox{ \ for any }A\in \mathcal{P}.
\end{equation}
{\color{black}To prove \eqref{inradius-cone-inYaoYao}, note that it follows exactly as in the proof of \eqref{V-lower-bound}, that $|A\cap B_2^n|\geq\delta_n$ for some computable constant $\delta_n$ that depends only on $n$. If $\Psi:\R^n\to\R^n$ is the linear map given by  $\Psi(v_i)=e_i$, $i=1,\dots,n$, then $\Psi K\supseteq \|\Psi^{-1}\|^{-1}_{\textnormal{op}}n^{-1}B_2^n$, thus  $\Psi(K\cap A)$ contains a ball of radius $r(\R^n_+\cap (\|\Psi^{-1}\|^{-1}_{\textnormal{op}}n^{-1}B_2^n))=:a$ and therefore $K\cap A$ contains a ball of radius $\|\Psi\|_{\textnormal{op}}^{-1}a$. Then, \eqref{inradius-cone-inYaoYao} follows immediately from Proposition \ref{prop-op}. }


For a fixed $u_0\in S^{n-1}$, we consider the measure $d\mu_0(x)=\langle x,u_0\rangle^2\mathbf{1}_K(x)\,dx$,  and a Yao-Yao {\color{black} equipartition $\mathcal{P}_0$ of $\R^n$ for $\mu_0$} based on $u_0^\bot$ and centered at the origin. We also fix an $A_0\in \mathcal{P}_0$ with $A_0\subseteq\{x\in\R^n:\langle u_0,x\rangle\geq 0\}$. It follows from \eqref{inradius-cone-inYaoYao} that there exists a $z_0$ such that 
(cf. \eqref{thm12-K-ball-in-out-isotropic})
\begin{equation}
\label{ball-in-A0-inYaoYao}
z_0+\frac{\tilde{r}}2\,B^n_2\subseteq A_0\cap K\subseteq nB^n_2\mbox{ \ and \ }\langle u_0,y\rangle\geq \tilde{r}\mbox{ \ for } y\in z_0+\frac{\tilde{r}}2\,B^n_2.
\end{equation}
We apply \eqref{thm-main-stab-nostrip} and \eqref{thm-main-stab-in-out-radius}  with $\Lambda=\frac{\tilde{r}}{nR}$ to $A_0\in \mathcal{P}_0$, and deduce that there exists an $o$-symmetric ellipsoid $E_0$ such that (cf. \eqref{thm12-K-ball-in-out-isotropic} and $R\geq n$)
\begin{align}
\label{thm12-main-stab-nostrip-E0}
\big|\big((K\Delta E_0)\cap A_0\big)\backslash {\color{black}\Theta_{u_0,\frac{\tilde{r}}{nR}}}\big|\leq &C_0\cdot\varepsilon^{\frac{1}{2}} \\
\label{thm12-main-stab-in-out-radius-E0}
\frac1R\cdot B^n_2\subseteq& E_0\cap K\subseteq E_0\subseteq RB^n_2
\end{align}
where $C_0>0$ depends only on $n$.

We consider the convex cone
\begin{equation}
\label{thm12-sigma0-def}
\sigma_0={\rm pos}\left(z_0+\frac{\tilde{r}}2\,B^n_2\right)\subseteq A_0. 
\end{equation}
If $x\in \sigma_0\cap \Theta_{u,\frac{\tilde{r}}{nR}}$, then $x=ty$ for a $y\in z_0+\frac{\tilde{r}}2\,B^n_2$ and $t\geq 0$, and hence $t\leq \frac{1}{nR}$ by \eqref{ball-in-A0-inYaoYao}. As $\|y\|\leq n$ also by \eqref{ball-in-A0-inYaoYao}, we deduce that $x\in \frac1R\cdot B^n_2\subseteq E_0\cap K$ by \eqref{thm12-main-stab-in-out-radius-E0}, thus
\eqref{thm12-main-stab-nostrip-E0} yields that
\begin{equation}
\label{thm12-K-E0-in-sigma0}
\big|(K\Delta E_0)\cap \sigma_0\big|\leq C_0\cdot\varepsilon^{\frac{1}{2}}.
\end{equation}

Let $w_1,\ldots,w_n\in S^{n-1}\cap z_0^\bot$ be the vertices of a regular $(n-1)$-simplex, and let $y_i=z_0+\frac{\tilde{r}}4\,w_i$ and $u_i=y_i/\|y_i\|$ for $i=1,\ldots,n$. Here, $y_i=t_0u_i$  for $t_0=\|y_i\|$, $i=1,\ldots,n$, and 
\begin{equation}\label{t0 bounds}
n\geq t_0> \|z_0\|\geq \frac{3}2\,\tilde{r}    
\end{equation}
hold by \eqref{ball-in-A0-inYaoYao}. It follows directly from \eqref{ball-in-A0-inYaoYao},  \eqref{thm12-sigma0-def} and \eqref{t0 bounds} respectively that 
\begin{align} 
\label{thm12-uiyi-ball-in-sigma0}
t_0u_i+\frac{\tilde{r}}4\,B^n_2\subseteq&\sigma_0\cap K&&\mbox{for }i=1,\ldots,n,\\
\label{thm12-uiyi-ball-far}
\langle u_i, y\rangle\geq &\tilde{r} && \mbox{for }y\in t_0u_i+\frac{\tilde{r}}4\,B^n_2.
\end{align}
 We claim that there exists $\tau>0$ depending only on $n$ such that
\begin{equation}
    \label{thm12-uiyi-capTheta}
\bigcap_{i=1}^n  \Theta_{u_i,\Lambda} \subseteq \tau\Lambda\,B^n_2\qquad\qquad\qquad\qquad \ \ \mbox{for any }\Lambda>0. \ \ \ \ 
\end{equation}

{\color{black}To prove \eqref{thm12-uiyi-capTheta}, we observe that by the construction of $u_1,\dots,u_n$ and \eqref{t0 bounds}, the regular $(n-1)$-dimensional simplex $\Sigma_0:=\textrm{conv}\{u_1,\dots,u_n\}$ contains an $(n-1)$-dimensional ball $B_0$ of radius depending only on $n$ hence, for the simplex $\Sigma={\rm conv}\{o,u_1,\ldots,u_n\}$ (again by \eqref{t0 bounds}) one has
$r(\Sigma)\geq r(\textrm{conv}(o,B_0))\geq\tilde{\tau}$  for some $\tilde{\tau}>0$ depending only on $n$.} 
Now, if $p_1,\ldots,p_n$ form the dual basis for $u_1,\ldots,u_n$; namely, $\langle p_i,u_i\rangle=1$ and $\langle p_i,u_j\rangle=0$ for $j\neq i$,  then $r(\Sigma)\geq \tilde{\tau}$ implies that $\|p_i\|\leq 1/\tilde{\tau}$ for $i=1,\ldots,n$, and hence if $x\in \bigcap_{i=1}^n  \Theta_{u_i,\Lambda}$, then
$$
\|x\|=\left\|\sum_{i=1}^n\langle x,u_i\rangle p_i\right\|\leq \frac{n}{\tilde{\tau}}\cdot\Lambda,
$$ 
proving \eqref{thm12-uiyi-capTheta} with $\tau=n/\tilde{\tau}$.

 Our core claim is that there exists $\widetilde{\Lambda}>0$ depending only on $n$ such that $\widetilde{\Lambda}\leq \frac1{\tau R}$, and for any $i=1,\ldots,n$, we have
\begin{equation}
\label{thm12-ui-keyclaim-Xialeph}
\big|(K\Delta E_0)\backslash \Theta_{u_i,\widetilde{\Lambda}}\big|\leq \widetilde{C}_0 \cdot\varepsilon^{\frac{1}{2}}
\end{equation}
where $\widetilde{C}_0>0$ depends only on $n$.

To prove \eqref{thm12-ui-keyclaim-Xialeph}, we consider  a Yao-Yao equipartition $\mathcal{P}_i$ of $\R^n$ for the measure $d\mu_i(x)=\langle x,u_i\rangle^2\mathbf{1}_K(x)\,dx$ based on $u_i^\bot$, centered at the origin and having an axis $v_i\in S^{n-1}$ satisfying that $\langle u_i,v_i\rangle>0$. We deduce from $t_0\leq n$ and 
\eqref{thm12-uv-close-dot-prod}  that
,  if $\varepsilon>0$ is small enough, \eqref{thm12-uiyi-ball-in-sigma0} and \eqref{thm12-uiyi-ball-far} yield that 
\begin{equation}
\label{ball-t0vi-inYaoYao}
t_0v_i+\frac{\tilde{r}}8\,B^n_2\subseteq K\cap \sigma_0\mbox{ \ and \ }\langle u_i,y\rangle\geq \tilde{r}\mbox{ \ for } y\in t_0v_i+\frac{\tilde{r}}8\,B^n_2.
\end{equation}
Let $A_1,\ldots,A_{2^{n-1}}$ be the elements of $\mathcal{P}_i$ containing $v_i$, and hence $t_0v_i\in A_j\cap K$ for each $A_j$. For any $j=1,\ldots,2^{n-1}$, there exists a $\tilde{z}_j\in A_j\cap K$ such that $\tilde{z}_j+\tilde{r}B^n_2\subseteq A_j\cap K$ by \eqref{inradius-cone-inYaoYao}. As $\tilde{z}_j+\tilde{r}B^n_2\subseteq t_0v_i+2nB^n_2$ according to \eqref{thm12-K-ball-in-out-isotropic},  \eqref{ball-t0vi-inYaoYao} implies that $z_j:=t_0v_i+\frac{\tilde{r}}{16n}(\tilde{z}_j-t_0v_i)$ satisfies  
$$z_j+\frac{\tilde{r}^2}{16n}\,B^n_2\subseteq K\cap \sigma_0\cap A_j\mbox{ \ and \ }\langle u_i,y\rangle\geq \tilde{r}\mbox{ \ for } y\in z_j+\frac{\tilde{r}^2}{16n}\,B^n_2.$$

We apply \eqref{thm-main-stab-nostrip} and \eqref{thm-main-stab-in-out-radius} to $A_j\in\mathcal{P}_i$ with 
$\Lambda=\widetilde{\Lambda}:=\min\left\{\frac{\tilde{r}}{nR},\frac{1}{\tau R}\right\}$ 
and deduce that there exists an $o$-symmetric ellipsoid $E_j$ such that (cf. \eqref{thm12-K-ball-in-out-isotropic} and $R\geq n$)
\begin{align}
\label{thm12-main-stab-nostrip-Ej}
\big|\big((K\Delta E_j)\cap A_j\big)\backslash {\color{black}\Theta_{u_i,\widetilde{\Lambda}}}\big|\leq &C'\cdot\varepsilon^{\frac{1}{2}} \\
\label{thm12-main-stab-in-out-radius-Ej}
\frac1R\cdot B^n_2\subseteq& E_j\cap K \subseteq E_j\subseteq RB^n_2
\end{align}
where $C'>0$ depends only on $n$.

We consider the convex cone
\begin{equation}
\label{thm12-sigmaj-def}
\sigma_j={\rm pos}\left(z_j+\frac{\tilde{r}}8\,B^n_2\right)\subseteq A_j\cap\sigma_0. 
\end{equation}
Repeating the proof of \eqref{thm12-K-E0-in-sigma0} with
\eqref{thm12-main-stab-nostrip-Ej} in place of \eqref{thm12-main-stab-nostrip-E0} and
\eqref{thm12-main-stab-in-out-radius-Ej} in place of \eqref{thm12-main-stab-in-out-radius-E0} we obtain
\begin{equation}
\label{thm12-K-Ej-in-sigmaj}
\big|(K\Delta E_j)\cap \sigma_j\big|\leq C'\cdot\varepsilon^{\frac{1}{2}}.
\end{equation}
Comparing \eqref{thm12-K-Ej-in-sigmaj} to \eqref{thm12-K-E0-in-sigma0}, and $\sigma_j\subseteq\sigma_0$ (cf. \eqref{thm12-sigmaj-def}) yield that

$$\big|(E_0\Delta E_j)\cap \sigma_j\big|\leq (C'+C_0)\cdot\varepsilon^{\frac{1}{2}}.
$$
In turn, we deduce from Lemma~\ref{EEprime-sandwiched}, \eqref{thm12-main-stab-in-out-radius-E0}, \eqref{thm12-main-stab-in-out-radius-Ej} and \eqref{thm12-sigmaj-def} that $|E_0\Delta E_j|\leq C^*\cdot\varepsilon^{\frac{1}{2}}$ for a $C^*>0$ depending only on $n$, and hence \eqref{thm12-main-stab-nostrip-Ej} yields
\begin{equation}
\label{thm12-main-stab-nostrip-KE0Aj}
\big|\big((K\Delta E_0)\cap A_j\big)\backslash \Theta_{u_i,\widetilde{\Lambda}}\big|\leq (C'+C^*)\cdot\varepsilon^{\frac{1}{2}}.
\end{equation}
As $\bigcup_{j=1}^{2^{n-1}}A_j=\{x\in \R^n:\langle x,u_i\rangle\geq 0\}$, and $K$ and $E_0$ are $o$-symmetric, it follows from {\color{black}\eqref{thm12-main-stab-nostrip-KE0Aj}} that
$$\big|(K\Delta E_0)\backslash \Theta_{u_i,\widetilde{\Lambda}}\big|\leq 2^n (C'+C^*)\cdot\varepsilon^{\frac{1}{2}},
$$
proving the core claim \eqref{thm12-ui-keyclaim-Xialeph} with $\widetilde{C}_0=2^n (C'+C^*)$ for any $i=1,\ldots,n$.

Now, $\bigcap_{i=1}^n  \Theta_{u_i,\widetilde{\Lambda}}\subseteq E_0\cap K$ by $\widetilde{\Lambda}\leq \frac1{\tau R}$, \eqref{thm12-uiyi-capTheta} and $\frac1R\,B^n_2\subseteq E_0\cap K$; therefore, our core claim \eqref{thm12-ui-keyclaim-Xialeph} implies $|K\Delta E_0|\leq n\widetilde{C}_0\varepsilon^{\frac{1}{2}}$, completing the proof of Theorem~\ref{Ball-ineq-stab}.
\end{proof}

\section{Proof of Theorem~\ref{BS-ineq-stab} and Collorary~\ref{Affine-ineq-stab}}
\label{secSantalo}

Comparing a symmetric convex body $L\subseteq\R^n$ with the centered ball of the same volume shows that 
\begin{equation}
\label{second-moment0}
|L|^{\frac{n+2}{n}}\leq \gamma_n\int_{L}\|x\|^2\,dx \mbox{ \ \ for \ }\gamma_n=\frac{(n+2)\omega_n^{\frac2n}}{n},
\end{equation}
where equality holds if and only if $L$ is a Euclidean ball. 

If $K$ is a symmetric convex body such that $K^\circ$ is in isotropic position, then it follows from first \eqref{second-moment0}, then applying $\|x\|^2=\sum_{i=1}^nx_i^2$ to $K$ and the definition \eqref{quasi-iso-u} of isotropicity 
to $L=K^\circ$ and $u=e_1,\ldots,e_n$, and after that \eqref{eq-thm-conj-n=2-proof}, and finally Theorem~\ref{thm-Ball-conj}, and then reversing the path, that
\begin{align}
\label{BS-Ball-K}
\left(|K|\cdot|K^\circ|\right)^{\frac{n+2}{n}}\leq &\gamma^2_n\int_{K}\|x\|^2\,dx\cdot \int_{K^\circ}\|x\|^2\,dx=
n\gamma^2_n\sum_{i=1}^n\int_{K} x_i^2\,dx\cdot \int_{K^\circ}x_i^2\,dx\\
\nonumber
=&n\gamma^2_n\int_K\int_{K^\circ}\langle x,y\rangle^2\, dx\, dy\leq n\gamma^2_n\int_{B_2^n}\int_{ B^n_2}\langle x,y\rangle^2\, dx\, dy\\
\label{BS-Ball-B}
=&\gamma^2_n\int_{B_2^n}\|x\|^2\,dx\cdot \int_{B_2^n}\|x\|^2\,dx=\left(\omega_n\cdot\omega_n\right)^{\frac{n+2}{n}}.
\end{align}

Comparing  \eqref{BS-Ball-K} and \eqref{BS-Ball-B} shows how Theorem~\ref{thm-Ball-conj} yields the Blaschke-Santal\'o inequality \eqref{eq-san}, but we also deduce the inequality \eqref{eq-moment-product} for the moment product. Actually, combining \eqref{BS-Ball-K} and \eqref{BS-Ball-B} with 
Theorem~\ref{Ball-ineq-stab0} directly yields the stability version Theorem~\ref{BS-ineq-stab} of the Blaschke-Santal\'o inequality.

 For Corollary~\ref{Affine-ineq-stab} concerning the affine surface area
 , we note that (cf. Lutwak \cite{Lut93b})
\begin{equation*}
\Omega(K)\leq n|K|^{\frac{n}{n+1}}|K^\circ|^{\frac{1}{n+1}}
\end{equation*}
In particular, Theorem~\ref{BS-ineq-stab} yields Corollary~\ref{Affine-ineq-stab}.

Finally, we show that the error term in Theorem~\ref{Ball-ineq-stab0}, Theorem~\ref{BS-ineq-stab} and Corollary~\ref{Affine-ineq-stab} is of optimal order. For symmetric convex bodies $K,M\subseteq \R^n$, the stability estimates in the statements above can be expressed in terms of the linear invariant quantity
$$
\widetilde{A}(K,M)=\min_{\Phi\in{\rm GL}(n)}A(K,\Phi M)
$$
where it is easy to see that we can write minimum in the definition. For example, the stability version Theorem~\ref{BS-ineq-stab} of the Blaschke-Santal\'o inequality can be written in the form
\begin{equation}
\label{BS-ineq-stab-tildeA}
|K|\cdot |K^\circ|\leq \left(1-\theta_n\cdot  \widetilde{A}(K,B^n_2)^2\right)\omega_n^2
\end{equation}
where $\theta_n>0$ depends only on $n$. 

\begin{example}[Optimality of Theorems~\ref{Ball-ineq-stab0} and \ref{BS-ineq-stab} and Corollary ~\ref{Affine-ineq-stab}] 

According to the argument above in Section~\ref{secSantalo}, it is sufficient to verify the optimality of the exponent $2$ of $\widetilde{A}(K,B^n_2)$ in the error term for the stability version \eqref{BS-ineq-stab-tildeA} of the Blaschke-Santal\'o inequality. In particular,
we prove that if $t\in(-\eta,\eta)$ for some $\eta>0$ depending on $n$, then there exists a family of convex bodies $K_t\subseteq\R^n$ and constants $c_0,C_0>0$ depending only on $n$ such that
\begin{description}
\item{(i)} $|K_t|=\omega_n$ and $|K_t\Delta B^n_2|\leq C_0|t|$;
\item{(ii)}  $|K_t\Delta E|\geq c_0|t|$ for any $o$-symmetric ellipsoid $E\subseteq\R^n$ with $|E|=\omega_n$;
\item{(iii)} $|K_t^\circ|\geq (1-C_0t^2)\omega_n$.
\end{description}
According to the properties (i)-(iii), the family $K_t$ satisfies that  $\widetilde{A}(K_t,B^n_2)$ tends to zero as $t$ tends to zero, and
$$
|K_t|\cdot |K_t^\circ|\geq \left(1-C\cdot \widetilde{A}(K_t,B^n_2)^2\right)\omega_n^2\mbox{ \ for \ }C=C_0\omega_n^2/c_0^2,
$$
verifying the optimality of the error term in the stability version \eqref{BS-ineq-stab-tildeA} of the Blaschke-Santal\'o inequality.

To construct $K_t$ if $|t|$ is small, 
let $u_0=\frac{e_1+\ldots+e_n}{\sqrt{n}}\in S^{n-1}$, and hence
$$
\|u_0-e_i\|>\frac12\mbox{ \ for }i=1,\ldots,n.
$$
We fix an even $C^\infty$ function $\varphi:S^{n-1}\to[0,1]$ such that
\begin{align*}
\varphi(u)=&0\mbox{ \ if }\|u-u_0\|\geq \frac14~~\textnormal{and}~~\|u-(-u_0)\|\geq\frac14\\
\varphi(u)=&1\mbox{ \ if }\|u-u_0\|\leq \frac1{8}.
\end{align*}

Now, there exists $\eta\in(0,1)$ such that if  $t\in[-\eta,\eta]$, then the even $C^\infty$ function $h_t=1+t\varphi$ on $S^{n-1}$ satisfies that  the symmetric matrix $\nabla^2 h_t+h_tI$ is positive definite
where $\nabla^2 h_t$ is the spherical Hessian and $\nabla h$ is the spherical gradient with respect to a moving frame, and hence there exists a symmetric convex body $\widetilde{K}_t\subseteq\R^n$ with $C^\infty_+$ boundary such that for any $u\in S^{n-1}$, we have
$$
h_{\widetilde{K}_t}(u)= h_t(u),\mbox { \ and  the $x\in\partial \widetilde{K}_t$ with exterior normal $u$ is }x=\nabla h_t(u)+h_t(u)u, 
$$ 
where $h_L$ stands for the support function of a convex body $L$.
We also consider the convex cones $\sigma={\rm pos}\{u\in S^{n-1}:\|u-e_1\|\leq \frac14\}$ and $\sigma_0={\rm pos}\{u\in S^{n-1}:\|u-u_0\|\leq \frac18\}$.
Since 
\begin{align*}
|\widetilde{K}_t|=&\frac1n\int_{S^{n-1}}h_t\det(\nabla^2 h_t+h_tI)\,d\mathcal{H}^{n-1},\\
|\widetilde{K}_t^\circ|=&\frac1n\int_{S^{n-1}}h_t^{-n}\,d\mathcal{H}^{n-1},
\end{align*}
we deduce that $|\widetilde{K}_t|$ and $|\widetilde{K}_t^\circ|$ are $C^\infty$ functions of $t\in(-\eta,\eta)$; moreover,
\begin{align*}
(1-|t|)B^n_2\subseteq &\widetilde{K}_t\subseteq (1+|t|)B^n_2,\\
h_{\widetilde{K}_t}(u)=\varrho_{\widetilde{K}_t}(u)= &1\mbox{ \ if }u\in \sigma\cap S^{n-1},\\
h_{\widetilde{K}_t}(u)=\varrho_{\widetilde{K}_t}(u)= &1+t\mbox{ \ if }u\in \sigma_0\cap S^{n-1}.
\end{align*}
After possibly choosing $\eta>0$ smaller, the renormalized body
$$
K_t=\lambda_t\cdot \widetilde{K}_t \mbox{ \ for }\lambda_t=\frac{\omega_n^{\frac1n}}{|\widetilde{K}_t|^{\frac1n}}\in\left(\frac34,\frac43\right)
$$
satisfies that $|K_t|=\omega_n$, $\frac12 B^n_2\subseteq K_t\subseteq 2 B^n_2$; moreover, $\lambda_t$ and $f(t)=|K_t^\circ|$ are $C^\infty$ functions of $t\in(-\eta,\eta)$, $\lambda_t=1+O(|t|)$ and hence 
\begin{align}
\nonumber
\left(1-O(|t|)\right)B^n_2\subseteq & K_t\subseteq \left(1+O(|t|)\right)B^n_2,\\
\label{hKt-in-sigma}
h_{K_t}(u)=\varrho_{K_t}(u)= &\lambda_t\mbox{ \ if }u\in \sigma\cap S^{n-1},\\
\label{hKt-in-sigma0}
h_{K_t}(u)=\varrho_{K_t}(u)= &(1+t)\lambda_t\mbox{ \ if }u\in \sigma_0\cap S^{n-1}.
\end{align}

We are ready to prove that $K_t$ satisfies (i)-(iii). Here (i) readily holds. For (iii), we observe that the Blaschke-Santal\'o inequality yields that $f$ has a maximum at $t=0$, and hence $f'(0)=0$. In turn, the Taylor formula $f(t)=f(0)+f'(0)t+\frac12\,f''(\xi)t^2$ where $|\xi|\leq |t|$ yields that
$f(t)\geq f(0)-C_0t^2$ for a constant $C_0>0$ independent of $t$, verifying (iii).

Finally, we prove (ii) indirectly; therefore, we suppose that there exists a sequence $t_k\neq 0$ tending to zero as $k$ tends to infinity and a sequence of $o$-symmetric ellipsoids $E_{t_k}$ such that  $|E_{t_k}|=\omega_n$ and
$$
\left|K_{t_k}\Delta E_{t_k}\right|= o(|t_k|)
$$
as $k$ tends to infinity, and seek a contradiction. We deduce from \eqref{hKt-in-sigma} and \eqref{hKt-in-sigma0} that
\begin{align*}
\Big|\big([\lambda_{t_k}B^n_2]\Delta E_{t_k}\big)\cap \sigma\Big|=&\Big|\left(K_{t_k}\Delta E_{t_k}\right)\cap \sigma\Big|= o(|t_k|),\\
\Big|\big(\left[(1+t_k)\lambda_{t_k}B^n_2\right]\Delta E_{t_k}\big)\cap \sigma_0\Big|=&\Big|\left(K_{t_k}\Delta E_{t_k}\right)\cap \sigma_0\Big|= o(|t_k|).
\end{align*}
In turn, Lemma~\ref{EEprime-sandwiched} yields that
$$
\Big|[\lambda_{t_k}B^n_2]\Delta E_{t_k}\Big|= o(|t_k|) \mbox{ \ and \ }
\Big|\left[(1+t_k)\lambda_{t_k}B^n_2\right]\Delta E_{t_k}\Big|=o(|t_k|),
$$
and hence the triangle inequality for the symmetric difference metric and $\lambda_{t_k}\in(\frac34,\frac43)$ imply that $\big|B^n_2\Delta \left[(1+t_k)B^n_2\right]\big|= o(|t_k|)$. However, $\big|B^n_2\Delta \left[(1+t_k)B^n_2\right]\big|\geq |t_k|\cdot \omega_n$, and this contradiction finally verifies (ii).

A similar computation, which we omit, shows the optimality of the exponent in Corollary ~\ref{Affine-ineq-stab} as well.
\end{example}

\noindent{\bf Acknowledgement: } B\"or\"oczky's research is supported in part by NKFIH NKKP grant 150613.

\noindent K\'aroly J. B\"or\"oczky, HUN-REN Alfr\'ed R\'enyi Institute of Mathematics, Budapest, Hungary\\
boroczky.karoly.j@renyi.hu\\

\noindent Konstantinos Patsalos, Department of Mathematics, University of Ioannina, Greece\\
k.patsalos@uoi.gr\\

\noindent Christos Saroglou, Department of Mathematics, University of Ioannina, Greece\\
csaroglou@uoi.gr \ \& \ christos.saroglou@gmail.com


\begin{thebibliography}{999}

\bibitem{And06}
B. Andrews:
Contraction of convex hypersurfaces by their affine normal. 
J. Differ. Geom., 43 (1996), no. 2, 207--230. 

\bibitem{AGM15}
S. Artstein-Avidan, A. Giannopoulos, V.D. Milman:
Asymptotic geometric analysis. Part I. 
Mathematical Surveys and Monographs, 202. American Mathematical Society, Providence, RI, 2015.
	
\bibitem{AGM21}
S. Artstein-Avidan, A. Giannopoulos, V.D. Milman: 
Asymptotic geometric analysis. Part II. 
Mathematical Surveys and Monographs, 261. American Mathematical Society, Providence, RI, 2021.
	

\bibitem{AKM04}
S. Artstein-Avidan, B. Klartag, V.D. Milman:
On the Santal\'o point of a function and a
functional Santal\'o inequality.
Mathematika, 54 (2004), 33--48.




\bibitem{Bal}
K. Ball:
Isometric problems in $\ell_p$ and sections of convex sets.
PhD thesis, University of Cambridge, 1986.

\bibitem{bal2} K. Ball: Some remarks on the geometry of convex sets. In: Lindenstrauss, J., Milman, V.D. (eds) Geometric Aspects of Functional Analysis. Lecture Notes in Mathematics, vol 1317. (1988) Springer, Berlin, Heidelberg. 

\bibitem{BaB11} 
K. Ball, K. B\"or\"oczky:
Stability of some versions of the Prékopa-Leindler inequality. 
Monatsh. Math., 163 (2011), no. 1, 1--14. 



\bibitem{Bor10} 
K. B\"or\"oczky:
Stability of the Blaschke-Santal\'o and the affine isoperimetric inequality.
Adv. Math., 225 (2010), no. 4, 1914--1928. 

\bibitem{DoH95}
G. Dolzmann, D. Hug: Equality of two representations of extended affine surface area. Arch. Math. (Basel), 65 (1995), 352--356.


\bibitem{Du} {S. Dubuc}: {Crit\`{e}res de convexit\'{e} et in\'{e}galit\'{e}s int\'{e}grales. }Ann. Inst. Fourier,
27 (1) (1977),  135--165.

\bibitem{FrM}
M. Fradelizi, M. Meyer:
Some functional forms of Blaschke-Santal\'o inequality.
Math. Z., 256 (2007), 379--395.

\bibitem{FMP10}
A. Figalli, F. Maggi, A. Pratelli:
A mass transportation approach to quantitative isoperimetric inequalities. 
Invent. Math., 182 (2010), 167--211. 

\bibitem{FMP08}
N. Fusco, F. Maggi, A. Pratelli:
The sharp quantitative isoperimetric inequality. 
Ann. of Math., (2), 168 (2008), 941--980.

\bibitem{FHT25+}
A. Figalli, P. van Hintum, M. Tiba: Sharp Quantitative Stability for the Prékopa-Leindler and Borell-Brascamp-Lieb Inequalities. arXiv:2501.04656
	

\bibitem{HL} Q. Huang, A. Li: The functional version of the Ball inequality. Proc. Amer. Math. Soc., 145
(2017), 3531--3541.

\bibitem{Iva15}
M. Ivaki:
Stability of the Blaschke-Santal\'o inequality in the plane. 
Monatsh. Math., 177 (2015), 451--459. 


\bibitem{KaS}
P. Kalantzopoulos, C. Saroglou:
On a $j$-Santal\'o Conjecture. Geom. Dedicata, 217 (2023) (article \# 91).

\bibitem{KLS95}
	R.~Kannan, L.~Lov\'asz, M.~Simonovits:
	Isoperimetric problems for convex bodies and a localization  lemma.
	Discrete Comput. Geom., 13 (1995), 541-559.
	

%


\bibitem{Leh09a} J. Lehec: On the Yao-Yao partition theorem. Arch. Math. (Basel), 92 (2009), 366--376.

\bibitem{Leh09b}
J. Lehec: 
Partitions and functional Santal\'o inequalities.  
Arch. Math. (Basel), 92  (2009), 89--94.

\bibitem{Leh09c}
J. Lehec: 
A direct proof of the functional Santal\'o inequality. C. R. Math. Acad. Sci. Paris, 347 (2009), no. 1-2, 55--58.

\bibitem{Le} L. Leindler: On a certain converse of H\"older’s inequality. II, Acta Sci. Math. (Szeged) 33 (1972), 217--223.

\bibitem{Lud10}
M. Ludwig: General affine surface areas. Adv. Math., 224 (2010), 2346--2360.

\bibitem{Lut91}
E. Lutwak: Extended affine surface area. Adv. Math., 85 (1991), 39--68.
	
\bibitem{Lut93b} E. Lutwak: Selected affine isoperimetric inequalities.  In: Handbook of convex geometry, Vol. A, B, 151--176, North-Holland, Amsterdam, 1993.
	
\bibitem{LZ} E. Lutwak, G. Zhang: Blaschke-Santal\'o inequalities. J. Differ. Geom., 47(1) (1997), 1--16. 

\bibitem{MeP}
M. Meyer, A. Pajor: On the Blaschke-Santal\'o inequality. Arch. Math. (Basel), 55 (1990), 82--93.

\bibitem{Mi-Pa} V. D. Milman, A. Pajor: {Isotropic position and inertia ellipsoids and zonoids of the unit ball of a normed $n$-dimensional space}. In: Lindenstrauss J., Milman V.D. (eds) Geometric Aspects of Functional Analysis. Lecture Notes in Mathematics, vol 1376. (1989) Springer, Berlin, Heidelberg. 

\bibitem{Pr}A. Pr\'ekopa: On logarithmic concave measures and functions. Acta Sci. Math. (Szeged), 34 (1973), 335--343.

\bibitem{San}
L. Santal\'o:
An affine invariant for convex bodies of $n$-dimensional space. Port. Math., 8 (1949), 155--161 (in Spanish).


\bibitem{Sch14}
R. Schneider: Convex Bodies: The Brunn-Minkowski Theory, second edition. Cambridge University Press, 2014.

\bibitem{ScW90}
C. Sch\"utt, E. M. Werner: The convex floating body. Math. Scand., 66 (1990), 275--290.


\bibitem{YY}A. C. Yao and F. F. Yao: A general approach to $d$-dimensional geometric queries. In: Proceedings of the seventeenth annual ACM symposium on Theory of computing, ACM Press (1985) 163--168.

	


	


\end{thebibliography}
\end{document}